\documentclass[11pt]{amsart}
\usepackage{}
\usepackage{bbm}
\usepackage[a4paper,left=4.5in, right=4.5in, top=7.5in, bottom=7.5in]{geometry}
\usepackage{mathrsfs}
\usepackage{CJK,fancyhdr,amscd}
\usepackage{anysize}
\usepackage{amsmath,amssymb,amsfonts}
\usepackage{epsfig}
\usepackage{amsthm}
\usepackage{hyperref}
\usepackage{indentfirst, latexsym, bm}
\usepackage{graphics}
\usepackage{bbding}
\usepackage{dsfont}
\usepackage[all,poly,knot]{xy}
\numberwithin{equation}{section}

\newtheorem{theo} {Theorem} [section]
\newtheorem{prop}[theo]{Proposition}
\newtheorem{coro}  [theo]     {Corollary}
\newtheorem{lemm}  [theo]     {Lemma}

\newtheorem{rema}  [theo]     {Remark}
\newtheorem{defi}  [theo]     {Definition}

\newcommand{\lk}{\left(}
\newcommand{\rk}{\right)}
\newcommand{\ii}{\mathds{1}}

\begin{document}
\pagestyle{fancy}
\renewcommand{\headrulewidth}{0pt}

\fancyhf{}

\fancyhead[OL]{\leftmark}

\fancyhead[ER]{New proof of Torelli Theorems on the moduli space of
Riemann Surfaces}

\fancyhead[OR,EL]{$\cdot$\ \thepage\ $\cdot$}

\renewcommand{\sectionname}{}
\renewcommand{\sectionmark}[1]{\markboth{}{\thesection\ #1}}
\renewcommand{\sectionmark}[1]{\markright{\thesection\ #1}{}}

\title{New proofs of the Torelli theorems for Riemann Surfaces}

\author{Kefeng Liu}
\address{Department of Mathematics, University of California at Los Angeles,
 Angeles, CA 90095-1555, USA; Center of Mathematical Sciences, Zhejiang
University, Hangzhou, China}
\email{liu@math.ucla.edu, liu@cms.zju.edu.cn}

\author{Quanting Zhao}
\address{Center of Mathematical Sciences, Zhejiang University, Hangzhou 310027, China}
\email{zhaoquanting@126.com}

\author{Sheng Rao}
\address{Center of Mathematical Sciences, Zhejiang University, Hangzhou 310027, China}
\email{likeanyone@zju.edu.cn}


\subjclass[2010]{Primary 14C34; Secondary 32G15, 32G05}

\date{\today}


\keywords{Torelli problem, Moduli space of Riemann Surfaces,
Teichm\"uller theory, Deformations of complex structures}

\begin{abstract}
In this paper, by using the Kuranishi coordinates on the
Teichm\"uller space and the explicit deformation formula of holomorphic
one-forms on Riemann Surface, we give an explicit expression of the
period map and derive new differential geometric proofs of the Torelli
theorems, both local and global, for Riemann Surfaces.
\end{abstract}

\maketitle

\section{Introduction}
The theme of this paper is to present a new differential geometric
understanding of the Torelli problems of Riemann Surfaces, which are
central topics in the study of the complex structures of Riemann
Surface. The Torelli problems are usually divided into two types:
local Torelli and global Torelli. These two problems are about the
immersion and injectivity of the period map from the moduli space of
Riemann Surfaces to the moduli space of principally polarized
abelian varieties, respectively.

Two key points of this paper are the use of the Kuranishi
coordinates on the Teichm\"{u}ler space $\mathcal {T}_g$ of Riemann
Surface of genus $g$ and the explicit deformation formula of
holomorphic one-forms in Section \ref{Kc}. Roughly speaking, the
Kuranishi coordinate chart of $\mathcal {T}_g$ is given by
\[ \begin{array}{ccc}
(B,b_0) & \rightarrow & \mathcal{T} _g \\
t & \rightarrow & [X_t,[F_t]],
\end{array} \]
where the triple $(\varpi,\varphi,F)$ is the Kuranishi family of
Riemann Surface with the Teichm\"{u}ller structure of $(X_0,[F_0])$. Let
us write $(B,b_0)$ as $\Delta_{p,\epsilon}$, where $p$ denotes the
point $[X_0,[F_0]] \in \mathcal {T}_g$. Then, given a global
holomorphic one-form $\theta \in H^0(X_p,\Omega_{X_p}^1)$ on $X_p$,
we have the following deformation formula $\theta (t)$ of $\theta$
for small $t$ on $X_t$:
\begin{equation*}
\begin{split}
\theta(t) & = \theta + \sum_{i=1}^n t_i \lk\mathbb {H} (\mu_i
\lrcorner \theta) + df_i\rk + \sum_{|I|\geq2} t^I \lk\sum_{j=1}^n
\mathbb {H}( \mu_j \lrcorner \eta_{(i_1,\cdots,i_j -1,\cdots,i_n)} )
+ df_{j,(i_1,\cdots,i_j -1,\cdots,i_n)}\rk,
\end{split}
\end{equation*}
where $\eta_{(i_1,\cdots,i_n)}$ is a sequence of $(1,0)$-forms on $X_p$,
$f_{j,(i_1,\cdots,i_j-1,\cdots,i_n)}\in C^{\infty}(X_p)$ and $|I| = \sum_{j=1}^{3g-3}i_j$.
Here and
henceforth $\mathbb{H}$ denotes the harmonic projection on
$(X_p,\omega_p)$, where $\omega_p$ is the Poinc\'{a}re metric on
$X_p$, and $n=3g-3$. An application of this to the canonical basis
$\{ \theta_p^{\alpha} \}_{\alpha=1}^g$ of $H^0(X_p,\Omega^1_{X_p})$
with respect to the symplectic basis $\{ A_{\gamma},B_{\gamma}
\}_{\gamma=1}^g$ for $\Delta_{p,\epsilon}$ tells us that
$$
\theta_{p}^{\alpha}(t)= \theta_{p}^{\alpha} + \sum_{i=1}^n t_i
\lk\mathbb {H} (\mu_i \lrcorner \theta^{\alpha}_{p}) + df^{\alpha}_i
\rk + \sum_{|I|\geq2} t^I \lk\sum_{j=1}^n \mathbb {H} ( \mu_j
\lrcorner \eta^{\alpha}_{(i_1,\cdots,i_j -1,\cdots,i_n)} ) +
df^{\alpha}_{j,(i_1,\cdots,i_j -1,\cdots,i_n)}\rk.
$$
The $g \times g$ matrix $A(t)$ is naturally defined as:
$$\sum_{|I|\geq1} t^I \lk\sum_{j=1}^n \mathbb {H} ( \mu_j \lrcorner
\eta^{\alpha}_{(i_1,\cdots,i_j -1,\cdots,i_n)}\rk
=A(t)^{\alpha}_{\beta} \bar{\theta}_{p}^{\beta}.$$ Meanwhile, let
$\pi_p$ be the $B$ period matrix of $\{ \theta^{\alpha}_p
\}_{\alpha=1}^g$. Then the period map $\Pi:\ \mathcal{ T }_g
\rightarrow \mathcal{ H }_g$ to the Siegel upper half space can be
written down explicitly:
$$
\Pi(t)= \begin{pmatrix}
\bar{\pi}_p & \pi_p \\
\ii_g  &  \ii_g \\
\end{pmatrix}\curvearrowright A(t)^{T},
$$
where the action $\curvearrowright$ is given by
$$\begin{pmatrix} C_1 & C_2 \\ C_3 & C_4 \\ \end{pmatrix}_{2g\times 2g} \curvearrowright Z = ( C_1 Z+ C_2 ) ( C_3 Z + C_4
)^{-1};$$ The transition formula between $\Pi(t)$ and $\Pi(\tau)$ of
two adjacent Kuranishi coordinates is
$$
\Pi(t) = L_{pq} \curvearrowright \Pi(\tau),
$$
where $L_{pq}$ is defined at the end of the proof of Theorem
\ref{Transt}.

Let $\Gamma_g$ be the mapping class group of Riemann Surface of
genus $g$, which has a natural representation in the symplectic
group $\mathrm{Sp} ( g , \mathbb {Z} )$ with integral coefficients,
written as $\rho:\, \Gamma_g \rightarrow \mathrm {Sp} ( g , \mathbb
{Z} )$. The moduli space $\mathcal {M}_g$ of Riemann Surfaces of
genus $g$ is the quotient space of $\mathcal {T}_g$ by $\Gamma_g$,
while $\mathcal {A}_g = \mathcal {H}_g / \mathrm {Sp} ( g , \mathbb
{Z} )$ is known as the moduli space of principally polarized abelian
varieties. In Section \ref{lTT}, we first give a proof of the
following two well-known local Torelli theorems by our deformation method.

\begin{theo} $a)$ \emph{(Local Torelli Theorem 1)}
The period map $\Pi: \mathcal{ T }_g \rightarrow \mathcal{ H }_g$ is
an immersion on the non-hyperelliptic locus and also when restricted
to the hyperelliptic locus for $g\geq3$; while for $g=2$, $\Pi$ is
an immersion on the whole $\mathcal { T }_g$ .

$b)$ \emph{(Local Torelli Theorem 2)}  For $g\geq2$, the period map
$\mathcal {J}: \mathcal {M}_g \rightarrow \mathcal {A}_g$ is an
immersion.
\end{theo}

Write  the quotient space of the Teichm\"{u}ller space
$\mathcal {T}_g$ by the Torelli group $T_g$ as $\mathcal {T}or_g$,
which has a natural
$\mathbb{Z}_2$ action. Recall that the Torelli group $T_g$ is the
kernel of the representation $\rho:\, \Gamma_g \rightarrow \mathrm
{Sp} ( g , \mathbb {Z} )$. Then we will present a new proof of the
following global Torelli theorem in Section \ref{gTT}:
\begin{theo}
$\mathcal {J}^{tor}: \, \mathcal {T}or_g / \mathbb {Z}_2 \rightarrow
\mathcal {H}_g$ is an embedding for $g\geq3$.
\end{theo}
We also prove that the period map $\Pi$ maps the $\Gamma_g$ orbit of
$\Delta_{p,\epsilon}$ onto the $\mathrm{Sp}(g,\mathbb{Z})$ orbit of
its image in $\mathcal {H}_g$. More precisely, let
$\Delta_{p,\epsilon}$ be a Kuranishi coordinate chart on $\mathcal{T}_g$
and $\Delta_{p,\epsilon}^{ [\phi] } :=[\phi]\Delta_{p,\epsilon}$ for
$[\phi]\in \Gamma_g$. Set
$$\rho([\phi])=\begin{pmatrix} U & V \\ R & S \\\end{pmatrix} \in \mathrm{Sp}( g,\mathbb{Z} ).$$
Then on $\Delta_{p,\epsilon}^{ [\phi] }$, the period map
$\widetilde{\Pi}(t)$ has the following relation with $\Pi(t)$:
$$\widetilde{\Pi}(t)= \begin{pmatrix} S & R \\ V & U \\ \end{pmatrix}\curvearrowright
\Pi(t).$$ Based on these, we prove that two $\Gamma_g$ orbits of
$\mathcal {M}_g$, if mapped to the same $\mathrm {Sp}(g,\mathbb{Z})$
orbit by $\mathcal {J}$, must coincide, and thus prove the main
result of this paper:
\begin{theo}[Torelli Theorem] The period map $\mathcal {J}: \, \mathcal {M}_g \rightarrow \mathcal {A}_g$
is injective for $g\geq2$.
\end{theo}

The maps considered in this paper can be summarized in the following
diagram:
$$
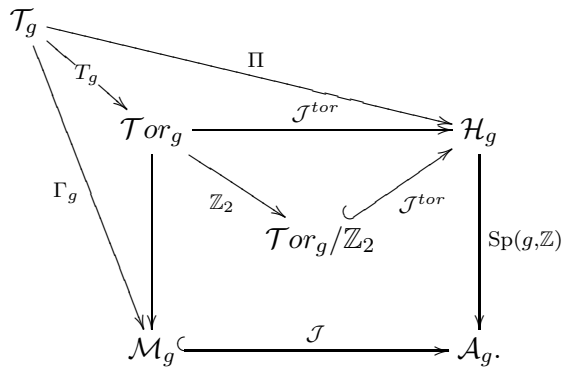
\xymatrix{ \mathcal {T}_g \ar[dddr]_{\Gamma_g}
\ar"2,4"^{\Pi}
\ar[dr]|-{T_g} \\
& \mathcal{T}or_g \ar[dr]_{\mathbb{Z}_2}\ar[dd] \ar[rr]^{\mathcal
{J}^{tor}}&
& \mathcal {H}_g \ar[dd]^{\mathrm{Sp} ( g,\mathbb{Z} )} \\
&&\mathcal{T}or_g/\mathbb{Z}_2\ar@{^{(}->}[ru]_{\mathcal {J}^{tor}}&\\
& \mathcal{M}_g \ar@{^{(}->}[rr]^{\mathcal {J}} && \mathcal{A}_g
.}$$

It is well-known that the global Torelli theorem holds by R.
Torelli's result \cite{T} and also the modern proofs \cite{A,Wei}
while the local Torelli holds due to the work of \cite{OS}. A more
complete list of the history about Torelli problems is contained in
the bibliographical notes on Page $261$ of \cite{ACGH}.

 \vspace{0.3cm}
$\mathbf{Acknowledgement}$ The authors dedicate this paper to Prof.
Andrey Todorov, who unexpectedly passed away in March 2012 during
his visit to Jerusalem. He had taught graduate courses at the Center
of Mathematical Sciences of Zhejiang University on deformation
theory and Hodge structures every summer in the recent years. The last
two authors would also like to express their gratitude to Dr.
Fangliang Yin, Prof. Fangyang Zheng, and Dr. Shengmao Zhu for many
inspirational discussions at CMS of Zhejiang University, and also to
Prof. Richard Hain for communications.

\section{Kuranishi Coordinates On $\mathcal {T}_g$}\label{Kc}

We first recall some basics of the construction of Kuranishi
coordinate charts, which is based on \cite{AC}. Fix a compact
topological surface $\Sigma$ of genus $g$ with $g\geq2$. The pair
$(C,[f])$ is a Riemann Surface $C$ with the Teichm\"{u}ller
structure $[f]$, where $f$ is an orientation-preserving
homeomorphism from $C$ to $\Sigma$ and $[f]$ denotes the isotopic
class represented by $f$. An isomorphism between Riemann Surfaces
with the Teichm\"{u}ller structures, $(C,[f])$ and $(C',[f'])$, is a
biholomorphic map $\phi$ from $C$ to $C'$ such that $[f]=[f'\phi]$.
The equivalence classes of all compact Riemann Surfaces of genus $g$
with the Teichm\"{u}ller structure, modulo the isomorphism
equivalences, actually constitute the Teichm\"{u}ler space $\mathcal
{T}_g$ of Riemann Surfaces of genus $g$. Thus an isomorphism class
of $[C,[f]]$ is a point in $\mathcal{T}_g$.

From the construction of Hilbert scheme, the existence of the Kuranishi
family of Riemann Surfaces follows. To be more precise, for every
Riemann Surface $C$, there exists a holomorphic deformation
$(\varpi,\varphi)$
\[ \begin{array}{cc} \varpi:\, \mathcal {X} \rightarrow B, & \varphi:\, C \stackrel{\simeq}{\rightarrow} X_{b_0} \\ \end{array} \]
of $C$ parametrized by a pointed base $(B,b_0)$, a complex manifold
with $\dim_{ \mathbb{C} } B= 3g-3$, and this deformation is
universal at $b_0$, actually universal at every point $b$ of $B$.
The pair $(\varpi,\varphi)$ is called the Kuranishi family of $C$.
For any other deformation $(\iota,\psi)$
\[ \begin{array}{cc} \iota:\, \mathcal {X}' \rightarrow B',\ \psi:\, C \stackrel{\simeq}{\rightarrow} X'_{b_0 '} \\ \end{array} \]
of $C$, there exists a unique map $(\phi,\Phi)$ in a small neighborhood of $b_0'$ such that the following diagram commutes
$$\xymatrix{\mathcal {X}' \ar[r]^{\Phi}\ar[d]_\iota& \mathcal {X}\ar[d]^\varpi\\
(B',b_0 ')\ar[r]^{\phi}  & (B,b_0),}$$ where
$\varphi^{-1}\Phi_{b_0'}\psi=\ii_{C}$ and $\mathcal {X'}$ is
isomorphic to the pullback family $\Phi^{\ast} \mathcal {X}$ on the
small neighborhood of $b_0'$. Accordingly, we also have a family of
Riemann Surfaces with the Teichm\"{u}ller structure $(X_b,[f_b])$,
i.e., $\varpi:  \mathcal {X} \rightarrow B$ together with local
topological trivialization $F^{\alpha}: \, \left. \mathcal {X}
\right|_{U_{\alpha}} \rightarrow \Sigma \times U_{\alpha}$, where
$\bigcup_{\alpha} U_{\alpha}$ is an open covering of $B$ such that
$[F^{\alpha}_b]=[f_b]$ with $b \in U_{\alpha}$. For any Riemann Surface with the Teichm\"{u}ller structure $(C,[f])$, Kuranishi
family also exists and satisfies exactly analogous universal
properties to the one without the Teichm\"{u}ller structure above.
Possibly shrinking $B$, we can describe the Kuranishi family of
$(C,[f])$ as a triple $(\varpi,\varphi,F)$ given by
\[ \begin{array}{ccc}
\varpi:\,  \mathcal {X} \rightarrow B, & \varphi:\, C
\stackrel{\simeq}{\rightarrow} C_{b_0}, & F:\, \mathcal {X}
\rightarrow \Sigma \times B,
\end{array} \]
where $F$ is a topological trivialization such that $F_{b_0}\varphi
= f$.

A Kuranishi coordinate chart of $\mathcal {T}_g$ is given by
\[ \begin{array}{ccc}
(B,b_0) & \rightarrow & \mathcal{T} _g \\
t & \rightarrow & [X_t,[F_t]],
\end{array} \]
where the triple $(\varpi,\varphi,F)$ is the Kuranishi family of
$(C,[f])$. From the classical Ehresmann's theorem, there is a
natural diffeomorphism $\Psi:\, X_{b_0} \times B \rightarrow
\mathcal {X}$; all the fibers of $\varpi:\, \mathcal {X} \rightarrow
B$
\[ \xymatrix{
\Sigma \times B & \\
\mathcal {X} \ar[u]^{F} \ar[r]^-{\Psi} & X_{b_0} \times B \\
} \] share the same differential structure as $X_{b_0}$. From this
point of view, for every $b \in B$,
the map $F_b \Psi_b^{-1}$ can be deformed to $F_{b_0} \Psi_{b_0}^{-1}$, i.e.
$[F_b \Psi_b^{-1}] = [F_{b_0} \Psi_{b_0}^{-1}]$. Let $\omega:\,
H_1(\Sigma,\mathbb{Z}) \times H_1(\Sigma,\mathbb{Z}) \rightarrow
\mathbb{Z}$ be the intersection pairing on $\Sigma$. The symplectic
basis of $H_1(\Sigma,\mathbb{Z})$ on $(\Sigma,\omega)$ gives, from
the map $\Psi F^{-1}$, one such basis on $X_{b_0}$, which is enjoyed
by the whole Kuranishi family $\mathcal {X}$ over the Kuranishi
coordinate chart $B$. Later on we will write $(B,b_0)$ as
$\Delta_{p,\epsilon}$, where $p$ denotes the point $[C,[f]]$ in
$\mathcal {T}_g$, and $\Delta_{p,\epsilon} = \{ t \in \mathbb {C}^n
\big| \| t \| <\epsilon,\ t(p)=0 \}$ with $n=3g-3$.

Fix the representation $\rho:\, \Gamma_g \rightarrow
\mathrm{Sp}(g,\mathbb{Z})$, where $\Gamma_g$ is the  mapping class
group, namely the isotopic classes of orientation preserving
homeomorphisms of $\Sigma$, and $\mathrm{Sp}(g,\mathbb{Z})$ is
actually $\mathrm{Aut} ( H_1(\Sigma,\mathbb{Z}),\omega )$. Now we have two Kuranishi
coordinate charts $\Delta_{p,\epsilon}$ and $\Delta_{q,\epsilon'}$
with $\Delta_{p,\epsilon} \cap \Delta_{q,\epsilon'} \neq \emptyset
$. Let $( \mathcal {X} , F )$ and $( \mathcal {Y} , G )$ denote the
two Kuranishi families with Teichm\"{u}ler structures over
$\Delta_{p,\epsilon}$ and $\Delta_{q,\epsilon'}$, respectively. Let
$r \in \Delta_{p,\epsilon} \cap \Delta_{q,\epsilon'}$. The
definition of Kuranishi coordinates tells us that
$[X_{t(r)},F_{t(r)}] = [Y_{\tau(r)},G_{\tau(r)}]$. Then we have a
biholomorphic map $\phi: X_{t(r)} \rightarrow Y_{\tau(r)}$ such that
$[F_{t(r)}]=[G_{\tau(r)} \phi]$. It is described in the following
picture
\[ \xymatrix{
\Sigma & & & \Sigma \\
X_p \ar[u]^{F_0} \ar[r]^-{\mathrm{diffeo}\Psi_{X}} & X_{t(r)} \ar[r]^-{\phi}
& Y_{\tau(r)} \ar[r]^-{\mathrm{diffeo}\Psi_{Y}} & Y_q \ar[u]^{G_0} \\
} \] that $[G_{0}\Psi_{Y}\phi\Psi_{X}F_{0}^{-1}]$ gives us an
element of $\Gamma_g$. From the representation $\rho$, a matrix in
$\mathrm{Sp}(g,\mathbb{Z})$ is obtained, linking
the two symplectic bases of the two Kuranishi coordinates.

\subsection{Small Deformation Of Holomorphic One-Forms}
Let $\Delta_{p,\epsilon}$ be a Kuranishi coordinate chart centered
at $p \in \mathcal{T}_g$ as above. Denote the corresponding
Kuranishi family on $\Delta_{p,\epsilon}$ by $\varpi:\, \mathcal {X}
\rightarrow \Delta_{p, \epsilon}$ with the central fiber
$\varpi^{-1}(p)=X_p$. Let $\theta \in H^0(X_p,\Omega_{X_p}^1)$ be a
global holomorphic one-form on $X_p$. We will construct an explicit
formula $\theta(t) \in H^0(X_t,\Omega_{X_t}^1)$, the holomorphic
deformation of $\theta$.

Denote the Poincar\'{e} metric on $X_p$ by $\omega_p$. Fix
$\{\mu_i\}^{n}_{i=1}$ as a basis of harmonic
$T_{X_p}^{(1,0)}$-valued $(0,1)$-form on $(X_p,\omega_p)$, namely
$\mathbb{H}^{0,1}_{\overline{\partial}}(X_p,T_{X_p}^{(1,0)})$. And
let $\mu(t)=\sum_{i=1}^{n}t_i \mu_i$ denote the Beltrami
differential of the Kuranishi family $\varpi:\, \mathcal {X}
\rightarrow \Delta_{p, \epsilon}$.

\begin{theo}\label{etat}
Given $\theta \in H^0(X_p,\Omega_{X_p}^1)$, there exists a unique
$(1,0)$-form $\eta(t)$ on $X_p$, which is holomorphic in $t$ for
sufficiently small $t$, satisfying
\begin{enumerate}
\item $\mathbb{H} \lk \eta(t) \rk=\theta$, where $\mathbb{H}$ is the harmonic projection on
$(X_p,\omega_p)$,
\item $\theta(t)=\lk\ii+\mu(t) \rk\lrcorner \eta(t) \in H^0(X_t,\Omega_{X_t}^1)$
\end{enumerate}
and $\theta(t)$ is the desired deformation of $\theta$.
\end{theo}

\begin{proof}
The formal power series of $\eta(t)\in A^{1,0}(X_p)$ can be written
out as
$$\eta(t)=\theta+\sum_{i=1}^n t_i\eta_i+\sum_{|I|\geq2}t^I\eta_I,$$
where $I=(i_1,\cdots,i_n)$, $t^I=t_1^{i_1}t_2^{i_2}\cdots t_n^{i_n}$ and $|I|=\sum_{j=1}^n i_j$.

Condition (1) implies
\begin{equation}\label{HV}
\begin{cases}
\mathbb{H}(\eta_i)=0,\\
\mathbb{H}(\eta_I)=0,\ |I|\geq 2.
\end{cases}
\end{equation}
Then, one has
\begin{align*}
\theta(t) &= \lk\ii + \mu(t) \rk \lrcorner \eta(t)\\
&=\lk\ii + \sum_{i=1}^n t_i\mu_i\rk\lrcorner\lk\theta+\sum_{j=1}^n  t_j \eta_j + \sum_{|J|\geq2} t^J \eta_J\rk\\
&= \theta + \sum_{i=1}^n t_i (\eta_i+\mu_i\lrcorner \theta) +
\sum_{|I|\geq2} t^I\lk \eta_{(i_1,\cdots,i_n)} + \sum_{k=1}^n \mu_k
\lrcorner \eta_{(i_1,\cdots,i_k -1,\cdots,i_n)}\rk.
\end{align*}
Since $\theta(t)$ is a holomorphic one-form on $X_t$ from Condition
(2), i.e., $d\theta(t)=0$, which implies
$$\begin{cases}
d(\eta_i+\mu_i\lrcorner \theta)=0,\\
d(\eta_{(i_1,\cdots,i_n)}+\sum_{k=1}^n \mu_k \lrcorner
\eta_{(i_1,\cdots,i_k -1,\cdots,i_n)})=0,
\end{cases}$$
we see that
\begin{equation}\label{PBE}
\begin{cases}
\overline{\partial} \eta_i+\partial{(\mu_i\lrcorner \theta)}=0,\\
\overline{\partial} \eta_{(i_1,\cdots,i_n)}+\partial{(\sum_{k=1}^n
\mu_k \lrcorner \eta_{(i_1,\cdots,i_k -1,\cdots,i_n)})}=0.
\end{cases}
\end{equation}
Combining with \eqref{HV} and solving the
$\overline{\partial}$-equation, we get
\begin{equation}\label{SPBE}
\begin{cases}
\eta_i=-\mathbb{G} \overline{\partial}^{*} \partial{(\mu_i\lrcorner \theta)},\\
\eta_{(i_1,\cdots,i_n)}=-\mathbb{G} \overline{\partial}^{*}
\partial {(\sum_{k=1}^n \mu_k \lrcorner \eta_{(i_1,\cdots,i_k
-1,\cdots,i_n)})}.
\end{cases}
\end{equation}
Here $\mathbb{G}$ denotes the Green operator in the Hodge
decomposition  with respect to the operator $\overline{\partial}$,
and $\mathbb
\ii=\mathbb{H}+(\overline{\partial}\overline{\partial}^*+\overline{\partial}^*\overline{\partial})\mathbb{G}$.
Thus we have proved the uniqueness of $\eta(t)$, which is fixed by
conditions (1) and (2).

Now let us discuss the convergence of the power series constructed
above. By the standard estimates of elliptic operators $\mathbb{G}$,
$\overline{\partial}^*$ and $\partial$, such as in \cite{MK}, we
easily have
$$\|\eta_I\|_{m+\alpha}\leq C^{|I|}\|\theta\|_{m+\alpha},$$
where the constant $C$ depends on $m$, $\alpha$ and $X_p$, and
$\|\cdot\|_{m+\alpha}$ is the H\"{o}lder norm. Consequently the
estimates of $\eta(t)$ yield
\begin{align*}
\| \eta(t) \|_{m + \alpha} &\leq \| \theta \|_{m+\alpha} + \| \theta \|_{m + \alpha} \sum_{\| I \| \geq1} C^{|I|} \epsilon^{|I|}\\
& = \| \theta \|_{m+\alpha} + \| \theta \|_{m + \alpha} \sum_{k\geq1} \sum_{\| I \|\geq k} C^{|I|} \epsilon^{|I|}\\
& = \| \theta \|_{m + \alpha} + \| \theta \|_{m+\alpha} \sum_{k\geq1} C^{k} \epsilon^k C^{k}_{n+k-1}\\
& \leq \| \theta \|_{m + \alpha} + \| \theta \|_{m + \alpha}
\sum_{k\geq1} C^{k} \epsilon^k n^{k},
\end{align*}
where $C^{k}_{n+k-1}$ is the common combinatorial number. By taking
$\epsilon$ smaller than $\frac{1}{2nC}$, we are done.
\end{proof}

\begin{coro}\label{thetat}
The deformation formula of $\theta$, with t small, is given by
\begin{equation*}
\begin{split}
\theta(t) & = \theta + \sum_{i=1}^n t_i \lk\mathbb {H} (\mu_i
\lrcorner \theta) + df_i \rk + \sum_{|I|\geq2} t^I \lk\sum_{j=1}^n
\mathbb {H} ( \mu_j \lrcorner \eta_{(i_1,\cdots,i_j -1,\cdots,i_n)}
) + df_{j,(i_1,\cdots,i_j -1,\cdots,i_n)}\rk
\end{split}
\end{equation*}
where $f_{j,(i_1,\cdots,i_j -1,\cdots,i_n)}\in C^{\infty}(X_p)$.
\end{coro}

\begin{proof}
From Theorem \ref{etat}, we can easily write out
\begin{eqnarray*}
\theta(t) & = & \lk\ii + \mu(t) \rk\lrcorner \eta(t) \\
& = & \Big(\ii + \sum_{i=1}^n  t_i\eta_i \Big) \lrcorner \Big(
\theta - \sum_{j=1}^n t_j \lk\mathbb{G} \overline{\partial}^*
\partial(\mu_j \lrcorner \theta) \rk
+ \sum_{|J| \geq 2} t^J\eta_J \Big)\\
& = & \theta + \sum_{|I|\geq1} t^I \lk\ii - \mathbb{G}
\overline{\partial}^* \partial \rk
\lk \sum_{j=1}^n \mu_j \lrcorner \eta_{(i_1,\cdots,i_j -1,\cdots,i_n)} \rk \\
& = & \theta + \sum_{|I|\geq1} t^I \lk\ii - \mathbb{G}
\overline{\partial}^* \partial \rk \lk \sum_{j=1}^n \mathbb {H} (
\mu_j \lrcorner \eta_{(i_1,\cdots,i_j -1,\cdots,i_n)} )+
\overline{\partial} f_{j,(i_1,\cdots,i_j -1,\cdots,i_n)} \rk\\
& = & \theta + \sum_{|I|\geq1} t^I \lk \sum_{j=1}^n \mathbb {H} (
\mu_j \lrcorner \eta_{(i_1,\cdots,i_j -1,\cdots,i_n)} )+
\overline{\partial} f_{j,(i_1,\cdots,i_j -1,\cdots,i_n)}
+ \partial \mathbb{G} \square_{\overline{\partial}} f_{ j,(i_1,\cdots,i_j -1,\cdots,i_n) } \rk\\
& = & \theta + \sum_{|I|\geq1} t^I \lk \sum_{j=1}^n \mathbb {H}
(\mu_j \lrcorner \eta_{(i_1,\cdots,i_j  -1,\cdots,i_n)} )+
df_{j,(i_1,\cdots,i_j -1,\cdots,i_n)} \rk.
\end{eqnarray*}
The convergence follows from Theorem \ref{etat}.
\end{proof}

\begin{rema} The iteration method to construct canonical forms on the
deformation space of Riemann Surfaces is essentially contained in
\cite{LSTY} and \cite[Theorem 2.1]{Y}. Note that \cite{LRY} contains
a generalization to K\"{a}hler manifolds while our proof emphasizes
on the uniqueness of the construction.
\end{rema}

Denote the canonical basis
 of $H^0(X_p,\Omega^1_{X_p})$ by $\{\theta^{\alpha}_{p} \}_{\alpha = 1}^g$
with respect to the symplectic basis $\{ A_{\gamma},B_{\gamma}
\}_{\gamma=1}^g$ on the Kuranishi coordinate chart
$\Delta_{p,\epsilon}$. Here a canonical basis means the unique
 basis of $H^0(X_p, \Omega^1_{X_p})$ such that its A period matrix is $\ii_g$. Let $\sigma_p$ and $\pi_p$ be the $A$
and $B$ period matrices of $\{ \theta^{\alpha}_p \}_{\alpha=1}^g$,
respectively, and $M_p=\mathrm{Im}(\pi_p)$. Applying the deformation
formula above, we get the holomorphic one-forms
$\theta_{p}^{\alpha}(t)$ on $X_t$, starting with
$\theta_{p}^{\alpha}$, given by
\begin{equation}\label{DF}
\begin{split}
\theta_{p}^{\alpha}(t) &= \theta_{p}^{\alpha} + \sum_{i=1}^n t_i
\lk\mathbb {H} (\mu_i \lrcorner \theta^{\alpha}_{p}) + df^{\alpha}_i
\rk + \sum_{|I|\geq2} t^I \lk\sum_{j=1}^n \mathbb {H} ( \mu_j
\lrcorner \eta^{\alpha}_{(i_1,\cdots,i_j -1,\cdots,i_n)} ) +
df^{\alpha}_{j,(i_1,\cdots,i_j -1,\cdots,i_n)} \rk.
\end{split}
\end{equation}

\begin{defi} Let $A(t)$ be a $g \times g$ matrix and $E(t)$ a $g \times 1$
vector given by:
\[\begin{cases}
\sum_{|I|\geq1} t^I \lk\sum_{j=1}^n \mathbb {H}
( \mu_j \lrcorner \eta^{\alpha}_{(i_1,\cdots,i_j -1,\cdots,i_n)}\rk=A(t)^{\alpha}_{\beta} \bar{\theta}_{p}^{\beta},\\
\sum_{|I|\geq1} t^I ( \sum_{j=1}^n df^{\alpha}_{ j,(i_1,\cdots,i_j
-1,\cdots,i_n) } )=E^{\alpha}(t).
\end{cases}\]
Also the homogeneous part of order $N$ of $A(t)$ is written as
$A_{N}(t)= \sum_{ |I| =N } t^{I} A_{I},$  $$\sum_{j=1}^n \lk \mathbb
{H} ( \mu_j \lrcorner \eta^{\alpha}_{(i_1,\cdots,i_j -1,\cdots,i_n)}
) \rk = A_{I},^{\alpha}_{\beta} \bar{\theta}^{\beta}_{p}.$$ In
particular, $\mathbb {H} ( \mu_i \lrcorner \theta^{\alpha}_p ) =
A_i,^{\alpha}_{\beta} \bar{\theta}^{\beta}_p$.
\end{defi}

Set
$$\Theta_p(t)=
\begin{pmatrix}
\theta^{1}_p(t) \\
 \vdots\\
 \theta^{g}_p(t)
\end{pmatrix}\ \textmd{and}\
\Theta_p= \begin{pmatrix}
 \theta^{1}_p \\
 \vdots\\
 \theta^{g}_p
\end{pmatrix}.
$$
 Thus by use of $A(t)$ and $E(t)$, we rewrite \eqref{DF} as
\begin{equation}\label{SDF}
\Theta_p(t)=
\begin{pmatrix}
 \ii_g & A(t)
\end{pmatrix}
\begin{pmatrix}
 \Theta_p \\
 \bar{\Theta}_p
\end{pmatrix}
 + E(t).
\end{equation}
Since a holomorphic one-form on Riemann Surfaces is uniquely
determined by its integration on $A$ cycles, it is clear that $\{
\theta^{\alpha}_p(t) \}_{\alpha=1}^g$ being a frame of
$H^0(X_t,\Omega^1_{X_t})$ on $X_t$, is equivalent to non-degeneration
of the $A$ period matrix $\sigma_{\alpha\beta}(t)$ on $X_t$ ,
i.e.,
\begin{equation}\label{Frame}
\begin{array}{ccc}
\det \lk \sigma_{\alpha\beta} (t) \rk = \det \lk \int_{A_{\alpha}}
\theta^{\beta}_p(t) \rk \neq 0 & \Longleftrightarrow & \det \lk
\ii_g + A(t)^{T} \rk \neq0,
\end{array}
\end{equation}
where $A(t)^{T}$ is the transpose of $A(t)$. And when $\{
\theta^{\alpha}_p (t) \}_{\alpha=1}^g$ becomes a frame, we have the
Hodge-Riemann bilinear relations on $X_t$
\[ \begin{cases}
0 = \frac{\sqrt{-1}}{2} \int_{X_t} \theta^{\alpha}_p (t) \wedge \theta^{\beta}_p (t),\\
0 < \frac{\sqrt{-1}}{2} \int_{X_t} \theta^{\alpha}_p (t) \wedge
\bar{\theta}^{\beta}_p (t),
\end{cases} \]
which and also \eqref{SDF} imply that
\[\begin{cases}
0 = \frac{\sqrt{-1}}{2} \int_{X_p} \lk \theta^{\alpha}_p +
A(t)^{\alpha}_{\gamma} \bar{\theta}^{\gamma} + E^{\alpha} (t) \rk
\wedge \lk \theta^{\beta}_p
+ A(t)^{\beta}_{\lambda} \bar{\theta}^{\lambda} + E^{\beta} (t) \rk,\\
0 < \frac{\sqrt{-1}}{2} \int_{X_p} \lk \theta^{\alpha}_p +
A(t)^{\alpha}_{\gamma} \bar{\theta}^{\gamma}_p + E^{\alpha} (t) \rk
\wedge \lk \bar{\theta}^{\beta}_p +
\overline{A(t)}^{\beta}_{\lambda} \theta^{\lambda}_p +
\bar{E}^{\beta} (t) \rk,
\end{cases} \]
and thus
\[\begin{cases}
0= M_p,_{\alpha\gamma} A(t)^{\beta}_{\gamma} - M_p,_{\beta\gamma} A(t)^{\alpha}_{\gamma}, \\
0 < M_p,_{\alpha\beta} - M_p,_{\lambda\gamma}
A(t)^{\alpha}_{\gamma} \overline{ A(t) }^{\beta}_{\lambda}.
\end{cases} \]
The matrix forms of these are given by
\begin{equation}\label{HRBR}
\begin{cases}
A(t)M_p = \lk A(t)M_p \rk^{T},\\
M_p -  A (t) M_p { \overline{A(t)} }^{T}> 0.
\end{cases}
\end{equation}
As our deformation formula is local, $\{ \theta^{\alpha}_p (t)
\}_{\alpha=1}^{g}$ is always a frame when $t \in
\Delta_{p,\epsilon}$ with $\epsilon$ sufficiently small.
Therefore, \eqref{Frame} and \eqref{HRBR} hold.

\subsection{Transition Formulas Between the Kuranishi Coordinates}
\begin{theo}\label{Transt}
Assume that the two Kuranishi coordinate charts $\Delta_{ p,
\epsilon }$ and $\Delta_{ q, \epsilon' }$ have a non-empty
intersection containing those two centers $p$ and $q$, and let $t$
and $\tau$ denote the corresponding
Kuranishi coordinates. Then $A(t)$ and $A(\tau)$ are related by the
following equality:
\begin{equation}\label{TranF}
\begin{split}
A(t)^{T} & = \begin{pmatrix}
           \bar{\pi}_p & \pi_p \\
           \ii_g  & \ii_g \\
           \end{pmatrix}^{-1}
           L_{pq}
           \begin{pmatrix}
           \bar{\pi}_q & \pi_q \\
           \ii_g & \ii_g \\
           \end{pmatrix} \curvearrowright A (\tau)^{T},
\end{split}
\end{equation}
where $L_{pq} \in \mathrm{Sp} ( g, \mathbb{Z} )$ denotes the
transition matrix between the symplectic bases of the two Kuranishi
coordinates in terms of transformations in $\mathrm{Sp} ( g,
\mathbb{Z} )$ of $\mathcal {H}_g$, and the action $\curvearrowright$
is given by
$$\begin{pmatrix} C_1 & C_2 \\ C_3 & C_4 \\ \end{pmatrix}
\curvearrowright Z = ( C_1 Z+ C_2 ) ( C_3 Z + C_4 )^{-1}.$$
\end{theo}

Observe that the transition matrix linking $A(t)$ and $A(\tau)$
depends only on $p$ and $q$, but not the coordinates $t$ and $\tau$.

\begin{proof}
From \eqref{SDF}, it yields
\begin{equation}\label{TranFatq}
\begin{split}
\left[ \Theta_p (q) \right] &= \begin{pmatrix}\ii_g & A \lk t(q) \rk
\end{pmatrix}
 \begin{pmatrix}
\left[ \Theta_p \right]\\
\left[ \bar{\Theta} _p \right]
\end{pmatrix},
\end{split}
\end{equation}
where $\left[ \Theta_p(q) \right]$ denotes the cohomology class
represented by $\Theta_p(q)$. The frames given by the deformation
formula $ \left[ \Theta_p (q) \right] $ and the canonical one $
\left[ \Theta_q \right] $ at $q$ are different by a multiple of a
nonsingular matrix $C$:
\begin{equation}\label{TranMatrix}
\begin{split}
\left[ \Theta_q \right] & = C \left[ \Theta_p (q) \right].
\end{split}
\end{equation}
Let $\{ A_{\gamma},B_{\gamma} \}_{\gamma=1}^g$ and $\{
A'_{\gamma},B'_{\gamma} \}_{\gamma=1}^g$ be the symplectic bases on
$\Delta_{p,\epsilon}$ and $\Delta_{q,\epsilon'}$, respectively. Set
$$\begin{pmatrix} A \\ B \\ \end{pmatrix} = \begin{pmatrix} A_1 \\ \vdots \\ A_g \\ B_1 \\ \vdots \\ B_g \\
\end{pmatrix}\ \textmd{and}\ \begin{pmatrix} A' \\ B' \\
\end{pmatrix}= \begin{pmatrix} A'_1 \\ \vdots \\ A'_g
\\ B'_1 \\ \vdots \\ B'_g \\ \end{pmatrix}.$$ Denote the matrix
linking these two bases by $\begin{pmatrix} U & V \\ R & S \\
\end{pmatrix} \in \mathrm{Sp}(g,\mathbb{Z})$, i.e.,
\begin{equation}\label{SymM}
\begin{pmatrix} A \\ B \\ \end{pmatrix} =
\begin{pmatrix} U & V \\ R & S \\ \end{pmatrix} \begin{pmatrix} A' \\ B' \\ \end{pmatrix}.
\end{equation}
By \eqref{SymM}, we integrate over $A$ cycles and $B$ cycles on
\eqref{TranMatrix} to get
\[ \begin{cases}
U + V\pi_q  =  \lk\ \ii_g + A \lk t(q) \rk^{T} \rk C^{T}, \\
R + S\pi_q  =  \lk\ \pi_p + \bar{ \pi }_p A \lk t(q) \rk^{T} \rk
C^{T},
\end{cases} \]
which imply that
\begin{equation}\label{CandAq}
\begin{cases}
C^{T} = \lk\ \ii_g + A \lk  t(q) \rk^{T} \rk^{-1} \lk U + V\pi_q \rk, \\
\lk\ \pi_p + \bar{ \pi }_p A \lk t(q) \rk^{T} \rk  \lk\ \ii_g + A
\lk t(q) \rk^{T} \rk^{-1} = \lk R + S\pi_q \rk \lk U + V\pi_q
\rk^{-1}.
\end{cases}
\end{equation}
By \eqref{TranFatq} and \eqref{TranMatrix}, we have
\begin{equation}\label{TranFpq}
\begin{split}
\left[ \Theta_q \right] &          = C \begin{pmatrix} \ii_g & A \lk
t(q) \rk
\end{pmatrix}
\begin{pmatrix}
\left[ \Theta_p \right]\\
\left[ \bar{\Theta} _p \right]\\
\end{pmatrix}             \\
&          =
\begin{pmatrix}
C & C A \lk t(q) \rk\\
\end{pmatrix}
\begin{pmatrix}
\left[ \Theta_p \right]\\
\left[ \bar{\Theta} _p \right]\\
\end{pmatrix}.
\end{split}
\end{equation}

Let $r \in \Delta_{ p , \epsilon} \cap \Delta_{ q , \epsilon'}$.
Then one has
\[ \left[ \Theta_p (r) \right] =
\begin{pmatrix}
\ii_g & A(t) \\
\end{pmatrix}
\begin{pmatrix}
\left[ \Theta_p \right]\\
\left[ \bar{\Theta} _p \right]\\
\end{pmatrix};\]
while by (\ref{TranFpq}), one also has
\[ \begin{split}
\left[ \Theta_p (r) \right] &       = C_r \left[ \Theta_q (r)
\right] = C_r \begin{pmatrix}
\ii_g & A(\tau) \\
\end{pmatrix}
\begin{pmatrix}
\left[ \Theta_q \right]\\
\left[ \bar{\Theta} _q \right]\\
\end{pmatrix}                   \\
& = C_r \begin{pmatrix}
\ii_g & A(\tau) \\
\end{pmatrix}
\begin{pmatrix}
C & C A \lk t(q) \rk\\
\overline { C A \lk t(q) \rk } & \overline { C } \\
\end{pmatrix}
\begin{pmatrix}
\left[ \Theta_p \right]\\
\left[ \bar{\Theta} _p \right]\\
\end{pmatrix},
\end{split} \]
where the two frames $ \left[ \Theta_p (r) \right]$ and $\left[
\Theta_q (r) \right]$ at the point $r$ are related by a nonsingular
matrix $C_r$. These give us the following identities:
\[ \begin{cases}
\ii_g = C_r \Big( C + A(\tau) \overline { C A \lk  t(q) \rk } \Big),\\
A(t) = C_r \Big( C A \lk t(q) \rk + A(\tau) \overline{ C } \Big).
\end{cases}
\]
Combine with \eqref{CandAq} to simplify the computation as follows:
\[   \begin{split}
A(t)^{T} & = \begin{pmatrix}
           \bar{\pi}_p & \pi_p \\
           \ii_g  & \ii_g \\
           \end{pmatrix}^{-1}
           \begin{pmatrix}
           \big(\ \bar{\pi}_p + \pi_p \overline { A \big( t(q) \big) }^{T} \big) \bar{C}^{T} & \big(\ \pi_p + \bar{\pi}_p  A \big( t(q) \big)^{T} \big) C^{T} \\
           \big(\ \ii_g + \overline { A \big( t(q) ) }^{T} \big) \bar{C}^{T} & \big(\ \ii_g + A \big( t(q) \big)^{T} \big) C^{T} \\
           \end{pmatrix}\curvearrowright A (\tau)^{T}        \\
           &{ = } \begin{pmatrix}
           \bar{\pi}_p & \pi_p \\           \ii_g  & \ii_g \\
           \end{pmatrix}^{-1}
           \begin{pmatrix}
           S & R \\
           V & U \\
           \end{pmatrix}
           \begin{pmatrix}
           \bar{\pi}_q & \pi_q \\
           \ii_g & \ii_g \\
           \end{pmatrix}\curvearrowright A (\tau)^{T},
\end{split}  \]
where $\begin{pmatrix} S & R \\ V & U \\ \end{pmatrix}$ also belongs
to $\mathrm{Sp}( g,\mathbb{Z} )$, denoted by $L_{pq}$.
\end{proof}
On our Kuranishi coordinate chart $\Delta_{p,\epsilon}$, the period
map $\Pi:\, \mathcal {T}_g \rightarrow \mathcal {H}_g$ can be
written out quite explicitly:
\[ \begin{split}
\Pi(t)_{\alpha\beta} &= \int_{B_{\alpha}} \sigma(t)^{\gamma\beta} \theta^{\gamma}_p (t)\\
& = \int_{B_{\alpha}} \sigma(t)^{\gamma\beta} \lk \theta^{\gamma}_p + A(t)^{\gamma}_{\delta} \bar{\theta}_p^{\delta} \rk\\
& = \pi_p,_{\alpha\gamma} \sigma(t)^{\gamma\beta} +
\bar{\pi}_p,_{\alpha\delta} A(t)^{\gamma}_{\delta}
\sigma(t)^{\gamma\beta},
\end{split} \]
where $\sigma(t)^{\alpha\beta}$ is the inverse matrix of $\sigma(t)_{\alpha\beta}$. \eqref{Frame} gives us
\begin{gather*}
\sigma_{\alpha\beta} (t) = \int_{A_{\alpha}} \theta^{\beta}_p(t) =
\lk \ii_g + A(t)^{T} \rk_{\alpha \beta}.
\end{gather*}
Now we can formulate these into the matrix form:
\begin{equation}\label{PM}
\begin{split}
\Pi(t) &= \lk \pi_p + \bar{\pi}_p A(t)^{T} \rk \lk \ii_g + A(t)^{T} \rk^{-1}\\
& = \begin{pmatrix}
\bar{\pi}_p & \pi_p \\
\ii_g  &  \ii_g \\
\end{pmatrix}\curvearrowright A(t)^{T}.
\end{split}
\end{equation}

\begin{coro}
The period maps $\Pi(t)$ and $\Pi(\tau)$ on the intersection of the
two Kuranishi coordinate charts $\Delta_{p,\epsilon}$ and
$\Delta_{q,\epsilon'}$ have the following transition formula
\begin{gather}
\Pi(t) = L_{pq} \curvearrowright \Pi(\tau).
\end{gather}
\end{coro}

\begin{proof}
By \eqref{PM} and Theorem \ref{Transt}, we have
\[\begin{split}
\Pi(t)     & = \begin{pmatrix}
               \bar{\pi}_p & \pi_p \\
               \ii_g &  \ii_g \\
               \end{pmatrix}\curvearrowright A(t)^{T}              \\
           & = \begin{pmatrix}
           \bar{\pi}_p & \pi_p \\
           \ii_g  & \ii_g \\
           \end{pmatrix}
           \begin{pmatrix}
           \bar{\pi}_p & \pi_p \\
           \ii_g  & \ii_g \\
           \end{pmatrix}^{-1}
           L_{pq}
           \begin{pmatrix}
           \bar{\pi}_q & \pi_q \\
           \ii_g & \ii_g \\
           \end{pmatrix}
           \begin{pmatrix}
           \bar{\pi}_q & \pi_q \\
           \ii_g  &  \ii_g \\
           \end{pmatrix}^{-1}\curvearrowright \Pi (\tau)        \\
           & = L_{pq}\curvearrowright \Pi (\tau).
\end{split} \]
\end{proof}

\section{Local Torelli Theorems and Matrix Model}\label{lTT}
\begin{theo}\emph{(Local Torelli Theorem 1)}\label{lTT1}
For $g\geq3$, the period map $\Pi:\, \mathcal{ T }_g \rightarrow
\mathcal{ H }_g$ is an immersion on the non-hyperelliptic locus
$\mathcal{T}_g - \mathcal {H}\mathcal {E}\mathcal {T}_g$ and also on
the hyperelliptic locus $\mathcal {H}\mathcal {E}\mathcal {T}_g$. In
the case $g=2$, $\Pi$ is an immersion on the whole $\mathcal { T
}_g$.
\end{theo}

\begin{proof}
From \eqref{PM}, the period map can be written as $\Pi(t)=\lk
\bar{\pi}_p A(t)^{T} + \pi_p \rk \lk A(t)^{T} + \ii_g\rk^{-1}$ via
Kuranishi coordinates. By use of $A(t)$, we expand it to obtain the
first order part $\Pi^{(1)}(t)$ of $\Pi(t)$:
\begin{equation}\label{FOP}
\begin{split}
\Pi^{(1)}(t) &= \bar{\pi}_p A^{T}_{1}(t) - \pi_p A^{T}_1(t)\\
&= -2\sqrt{-1} M_p A^{T}_1(t)\\
& = -2\sqrt{-1} \sum_{i=1}^n t_i M_p,_{\alpha\gamma} A_i,^{\beta}_{\gamma}\\
& = \sum_{i=1}^n t_i \int_{X_p} \theta_p^{\alpha} \wedge \mathbb {H}
( \mu_i \lrcorner \theta_p^{\beta} ) \\
&= \sum_{i=1}^n t_i
\int_{X_p} \theta_p^{\alpha} \wedge ( \mu_i \lrcorner
\theta_p^{\beta} ).
\end{split}
\end{equation}
It is a well-known fact that the pairing $\mathbb {H}^{(0,1)}_{
\overline{\partial}} ( X_p , T^{(1,0)}_{X_p}) \times H^0(X_p ,
2K_{X_p}) \rightarrow \mathbb{C}$ is non-degenerate. The matrices
$\{ \int_{X_p} \theta_p^{\alpha} \wedge ( \mu_i \lrcorner
\theta_p^{\beta} ) \}_{i=1}^n$ are linearly dependent if and only if
there exists a nonzero vector $t = ( t_1,\cdots,t_n )$ such that the
matrix
\[ \int_{X_p} \theta^{\alpha} \wedge \left( \sum_{i=1}^n t_i \mu_i \right) \lrcorner \theta^{\beta} \equiv 0. \]
This is equivalent to that the multiplication map $H^0
(X_p,K_{X_p}) \times H^0 (X_p,K_{X_p}) \rightarrow
H^0(X_p,2K_{X_p})$ is not surjective.

A well-known theorem by Max Noether in \cite[P. 117]{ACGH}  tells us
the multiplication map $H^0 (X_p,K_{X_p}) \times H^0 (X_p,K_{X_p})
\rightarrow H^0(X_p,2K_{X_p})$ is always surjective when $X_p$ is
non-hyperelliptic. Thus the period map $\Pi$ is an immersion when
restricted to $\mathcal{T}_g - \mathcal {H}\mathcal {E}\mathcal
{T}_g$ for $g\geq3$. As to the hyperelliptic case described in
\cite[P. 104]{FK}, the image of the multiplication map is exactly
the vector space $\left( H^0 (X_p, 2K_{X_p}) \right)^{J}$, namely,
the elements in $H^0(X_p,2K_{X_p})$ invariant under the action by
the hyperelliptic involution $J$ with $\dim_{\mathbb{C}}\left( H^0
(X_p, 2K_{X_p}) \right)^{J} = 2g - 1$. Also the tangent direction of
the hyperelliptic locus can be identified with $\left( H^1(X_p,
T_{X_p}) \right)^{J}$. Hence these directions can not be degenerate
and thus $\left. \Pi \right|_{\mathcal {H}\mathcal {E}\mathcal
{T}_g}$ is still an immersion for $g\geq3$. As we know, any Riemann Surface of genus $2$
is hyperelliptic and the above multiplication map is surjective
since $2g-1 = 3g-3$ when $g=2$. Consequently, $\Pi$ is an immersion
on $\mathcal {T}_g$ for $g=2$.
\end{proof}

\begin{defi}
$T_g$, $\widetilde{T}_g$ and $\mathcal {T}or_g$.

$T_g$, called the Torelli group, is the kernel of the representation
$\rho: \Gamma_g \rightarrow \mathrm{Sp} ( g,\mathbb{Z} )$ while the
extended Torelli group $\widetilde{T}_g$ is defined to be
$\rho^{-1}( \langle -\ii_{2g}
\rangle )$ where $\langle -\ii_{2g} \rangle$ is the subgroup of
$\mathrm{Sp}(g,\mathbb{Z})$ generated by $-\ii_{2g}$. The Torelli
space $\mathcal {T}or_g$ is the quotient space of the
Teichm\"{u}ller space $\mathcal {T}_g$ by $T_g$.
\end{defi}
\begin{defi}
$\Gamma_g(n)$ and $\mathcal {M}_g^{(n)}$.

$\Gamma_g(n)$, the level $n$ subgroup of the mapping class group
$\Gamma_g$, is the kernel of the representation $\Gamma_g \stackrel
{\rho} {\rightarrow} \mathrm{Sp} ( g,\mathbb{Z} ) \stackrel {\pi}
{\rightarrow} \mathrm{Sp} ( g,\mathbb{Z}_n )$. $\mathcal
{M}_g^{(n)}$ is the moduli space of Riemann Surfaces of genus $g$
with level $n$ structure, which is defined as the quotient space of
the Teich\"{m}uller space $\mathcal {T}_g$ by the group action of
$\Gamma_g(n)$. And we identify $\Gamma_g(1)$ with $T_g$.
\end{defi}

As we know, the action of the mapping class group $\Gamma_g$ on the
Teichm\"{u}ller space $\mathcal {T}_g$ is properly discontinuous.
From the construction of the Kuranishi coordinate of $\mathcal
{T}_g$ in \cite{AC}, we know that the isotropy group $\Gamma_g^p$ of
$\Gamma_g$ at $p=[X_p,[f_p]]$ on $\mathcal {T}_g$ is
$\mathrm{Aut}(X_p)$ if we fix the injective homomorphism
\[ \begin{array}{ccc}
\mathrm{Aut}(X_p) & \rightarrow & \Gamma_g \\
h                 &  \rightarrow & [f_p h f_p^{-1}].
\end{array} \]
Moreover, we can choose $\epsilon$ and $\epsilon'$ sufficiently
small such that the points $p$ and $p'$ in different $\Gamma_g$
orbits have disjoint Kuranishi coordinates, i.e.,
$\Delta_{p,\epsilon} \cap \Delta_{p',\epsilon'}=\emptyset$, and
\[ \{ \gamma \in \Gamma_g \Big| \gamma \Delta_{p,\epsilon} \cap \Delta_{p,\epsilon} \neq \emptyset \} = \Gamma_g^p. \]

\begin{prop}
The action of $T_g$ and $\Gamma_g(n)$ with $n\geq3$ on $\mathcal {T}_g$ is fixed point free.
\end{prop}

This proposition implies that $\mathcal{T}or_g$ and $\mathcal
{M}_g^{(n)}$ with $n\geq3$ are complex manifolds of complex
dimension $3g-3$.

\begin{proof}
We just need to show that $T_g \cap \Gamma_g^p = \{ 1 \}$ and
$\Gamma_g(n) \cap \Gamma_g^p = \{ 1 \}$. But we can identify
$\Gamma_g^p$ with $\mathrm{Aut}(X_p)$. It follows from the theory of
automorphism groups of Riemann Surfaces in \cite[Chapter V]{FK} that
the representation of $\mathrm{Aut}(X_p)$ in $H_1(X_p,\mathbb{Z})$
and $H_1(X_p,\mathbb{Z}_n)$ with $n\geq3$ are faithful, i.e., the
homomorphisms $\mathrm{Aut}(X_p) \rightarrow
\mathrm{Sp}(g,\mathbb{Z})$ and $\mathrm{Aut}(X_p) \rightarrow
\mathrm{Sp}(g,\mathbb{Z}_n)$ are injective. Now the isotropy group
$\Gamma_g^p$ embeds into $\mathrm{Sp}(g,\mathbb{Z})$ by the
representation $\rho:\, \Gamma_g \rightarrow \mathrm{Sp} (
g,\mathbb{Z} )$ if we view $\Gamma_g^p$ as $f_p^{-1}
\mathrm{Aut}(X_p) f_p$ while $T_g$ is the kernel of $\rho$. Thus
$T_g \cap \Gamma_g^p = \{ 1 \}$. Similarly, $\Gamma_g^p$ embeds into
$\mathrm{Sp}(g,\mathbb{Z}_n)$ by the representation $\Gamma_g
\stackrel {\rho} {\rightarrow} \mathrm{Sp} ( g,\mathbb{Z} )
\stackrel {\pi} {\rightarrow} \mathrm{Sp} ( g,\mathbb{Z}_n )$, and
$\Gamma_g(n)$ is the kernel of the representation $\pi\rho$. Finally
we have $\Gamma_g(n) \cap \Gamma_g^p = \{ 1 \}$.
\end{proof}

From the discussion above, we can shrink our Kuranishi coordinate
chart $\Delta_{p,\epsilon}$ on $\mathcal {T}_g$ such that $\gamma
\Delta_{p,\epsilon} \cap \Delta_{p,\epsilon} = \emptyset$ for any
$\gamma \in T_g$ and $\gamma \neq 1$. Naturally, the Kuranishi
coordinate chart $\Delta_{p,\epsilon}$ descends to $\mathcal {T}or_g
$. Let $\mathbb{Z}_2 \cong \widetilde {T}_g / T_g$ and then
$\mathcal {T}or_g$ has a natural $\mathbb{Z}_2$ action. There is a
commutative diagram
$$\xymatrix{ \mathcal{T}_g \ar[rd]^{\Pi}\ar[d]_{T_g}& \\
\mathcal{T}or_g \ar[r]_{\mathcal {J}^{tor}} & \mathcal {H}_g.}$$

\begin{lemm}\label{HY}
Let $X$ be a compact Riemann Surface with genus $g\geq2$ and $J$ an
involution on $X$, which does not fix any element of $H^0 (X,K_X)$.
Then $X$ is hyperelliptic and $J$ must be a hyperelliptic
involution.
\end{lemm}

\begin{proof}
Since $J^2=\ii$, the automorphism $J^{\ast}:\, H^0(X,K_X)
\rightarrow H^0(X,K_X)$ has two eigenvalue $\pm 1$. As $J^{\ast}$
fixes no element of $H^0 (X,K_X)$, $J^{\ast} = - \ii_g$ on $H^0
(X,K_X)$. Consider the quotient map $\pi: \,X \rightarrow X / J$, a
$2:1$ branched covering map, and $\pi=J \pi$. We claim that
$g(X/J)=0$. If not, there exists a nonzero holomorphic one-form
$\theta \in H^0(X/J,K_X)$. Pulling it back, we derive a nonzero
holomorphic one-form $\pi^{\ast} \theta \in H^0(X,K_X)$. But
$\pi^{\ast} \theta$ is invariant under $J^{\ast}$, which is a
contradiction. Thus $X/J$ is the Riemann sphere and $\pi$ is a
degree $2$ meromorphic function on $X$, which implies that $X$ is
hyperelliptic and $J$ is a hyperelliptic involution.
\end{proof}

\begin{prop}\label{Z2action}
$\mathbb{Z}_2$ acts freely on the non-hyperelliptic locus $\mathcal {T}or_g - \mathcal {H}\mathcal {E}\mathcal {T}or_g$
of $\mathcal{T}or_g$ and fixes every point in the hyperelliptic locus $\mathcal {H}\mathcal {E}\mathcal {T}or_g$ for $g\geq3$.
In the case $g=2$, $\mathbb{Z}_2$ acts trivially on $\mathcal{T}or_g$.
\end{prop}

\begin{proof}
Let $\{[\phi]\}$ be the non-unit element in $\mathbb{Z}_2$, where
$[\phi] \in \Gamma_g$ is a representative of the class $\{[\phi]\}$
and $\rho([\phi])=-\ii_{2g}$. Then that $\{[\phi]\}$ fixes a point
$\{[X_p,[f_p]]\}$ in $\mathcal {T}or_g$
 is equivalent to that there exists some element $[\psi] \in T_g$ such that $[X_p,[\phi f_p]]=[X_p,[\psi f_p]]$.
We have a commutative diagram up to isotopy
$$\xymatrix{ &\Sigma& \\
X_p\ar[ru]^{\phi f_p}\ar[rr]^h && X_p\ar[lu]_{\psi f_p} },$$ where
$h$ is an automorphism of $X_p$. Hence $h \simeq f_p^{-1} \psi^{-1}
\phi f_p$. As $\rho([\psi])=\ii_{2g}$ and $\rho([\phi])=-\ii_{2g}$,
$h_{\ast}:\, H_1(X_p,\mathbb{Z}) \rightarrow H_1(X_p,\mathbb{Z})$ is
nothing but $-\ii_{2g}$. Since a holomorphic one-form is uniquely
determined by its integration on A cycles and
$$\int_{A_{\alpha}} h^{\ast} \theta = \int_{h_{\ast}A_{\alpha}}
\theta = -\int_{A_{\alpha}} \theta,$$ $h^{\ast}:\, H^0(X_p,K_{X_p})
\rightarrow H^0(X_p,K_{X_p})$ is $-\ii_{g}$. Also the representation
of $\mathrm {Aut} ( X_p )$ to $H_1 ( X_p, \mathbb{Z} )$ is faithful
and hence $h$ is an involution. From Lemma \ref{HY}, $h$ is a
hyperelliptic involution and $X_p$ is hyperelliptic.
\end{proof}

It is easy to check that the $\mathbb{Z}_2$ orbit of $\mathcal
{T}or_g$ has the same image under $\mathcal {J}^{tor}$, since we
also have Kuranishi coordinate on $\mathcal {T}or_g$ and use
\eqref{PM}. Consequently, $\mathcal {J}^{tor}$ factors through
$\mathcal {T}or_g/\mathbb{Z}_2$:
$$\xymatrix{\mathcal{T}or_g \ar[dr]_{\mathbb{Z}_2}\ar[rr]^{\mathcal
{J}^{tor}}&& \mathcal {H}_g\\
&\mathcal{T}or_g/\mathbb{Z}_2\ar[ru]_{\mathcal {J}^{tor}}& }.$$ From
Proposition \ref{Z2action}, $\mathcal {T}or_g \rightarrow \mathcal
{T}or_g/\mathbb{Z}_2$ is a $2:1$ branched covering map branching
over the hyperelliptic locus $\mathcal {H}\mathcal {E}\mathcal
{T}or_g$ for $g\geq3$. Meanwhile, the Kuranishi coordinate chart
$\Delta_{p,\epsilon}$, $p \in \mathcal {T}or_g - \mathcal
{H}\mathcal {E}\mathcal {T}or_g$, also descends to $\mathcal
{T}or_g/\mathbb{Z}_2$. When $p \in \mathcal {H}\mathcal {E}\mathcal
{T}or_g$, we can view the Kuranishi coordinate $\Delta_{p,\epsilon}$
on $\mathcal {T}or_g$ as follows: $\Delta^{3g-3}$ decomposes into
$\Delta^{2g-1} \times \Delta^{g-2}$ where $\Delta^{2g-1}$ indicates
the direction of $T_{p} (\mathcal {H}\mathcal {E}\mathcal {T}or_g)$
and $\Delta^{g-2}$ is the normal direction in which the period map
$\mathcal {J}^{tor}_{\ast}$ vanishes. The $\mathbb{Z}_2$ action
fixes $\Delta^{2g-1}$ but acts as the multiplication of $-1$ on
$\Delta^{g-2}$. Thus $\mathcal {T}or_g / \mathbb{Z}_2$ locally looks
like $\Delta^{2g-1} \times (\Delta^{g-2}/\mathbb{Z}_2)$ around the
hyperelliptic locus.

\begin{theo}\emph{(Local Torelli Theorem 2)}\label{lTT2}
$\mathcal {J}: \, \mathcal {M}_g \rightarrow \mathcal {A}_g$ is an immersion for $g\geq2$.
\end{theo}

This local Torelli Theorem was first proved by F. Oort and J.
Steenbrink \cite{OS} and then by Y. Karpishpan \cite{Kar} under his
framework of understanding higher order derivatives of period map in
terms of \v{C}ech cohomology. We approach it by our deformation
method.

\begin{proof}
From the local Torelli Theorem \ref{lTT1}, the tangent map $\Pi_{\ast}$,
restricted to $\mathcal {T}_g - \mathcal {H}\mathcal {E}\mathcal
{T}_g$, is injective for $g\geq3$ and everywhere injective for
$g=2$. Thus it suffices to show that the tangent map of $\mathcal
{J}: \, \mathcal {M}_g \rightarrow \mathcal {A}_g$ at hyperelliptic
locus $\mathcal {H}\mathcal {E}_g$ is injective for $g\geq3$. To
this end, we lift the period map to $\mathcal {J}^{tor}: \mathcal
{T}or_g/ \mathbb{Z}_2 \rightarrow \mathcal {H}_g$. Fix $p \in
\mathcal {H}\mathcal {E}\mathcal {T}or_g$ which descends to
$\tilde{p}$ in $\mathcal {T}or_g / \mathbb {Z}_2$. From Proposition
3.7, $\tilde{p}$ is a double point. Moreover, the dimension of the
Zariski tangent space at $\tilde{p}$ is $\frac{g(g+1)}{2}$. In fact,
as $\mathcal {T}or_g$ is a complex manifold of complex dimension
$3g-3$ and $p$ is a smooth point, we can choose local parameters
$(t_1,t_2,\cdots,t_{3g-3})$ such that $\widehat{ \mathcal {O} }_{
\mathcal {T}or_g , p } = \mathbb{C}[[t_1,t_2,\cdots,t_{3g-3}]]$ and
$\mathbb{Z}_2$ action is given by
\[ \begin{cases}
\mathbb{Z}_2^{\ast} t_i =t_i, & 1 \leq i \leq 2g-1, \\
\mathbb{Z}_2^{\ast} t_i =-t_i, & 2g \leq i \leq 3g-3, \\
\end{cases}  \]
where $\{ t_i \}_{i=1}^{2g-1}$ indicates the tangent directions of
$\mathcal {H}\mathcal {E}\mathcal {T}or_g$ and  $\{ t_i
\}_{i=2g}^{3g-3}$ is the normal directions in which $\mathcal
{J}^{tor}_{\ast}$ vanishes. Clearly,
\[\widehat{ \mathcal {O} }_{ \mathcal {T}or_g/\mathbb{Z}_2 , \tilde{p} }
= \left( \widehat{ \mathcal {O} }_{ \mathcal {T}or_g , p }
\right)^{\mathbb{Z}_2} = \mathbb{C} [[
t_1,\cdots,t_{2g-1},t_{2g}^2,t_{2g}t_{2g+1},\cdots,t_{3g-3}^{2} ]].
\] It is exactly the $\frac{g(g+1)}{2}$ parameters that give the
basis of the Zariski tangent space at $\tilde{p}$. We denote these
directions by $\{ D_{k} , D_{ij} \}_{1 \leq k \leq 2g-1 , 2g \leq i
\leq j \leq 3g-3 }$, respectively. Also by \eqref{PM}, we know that
$\mathcal {J}^{tor}$ can also be written as $\lk \bar{\pi}_p
A(t)^{T} + \pi_p \rk \lk A(t)^{T} + \ii_g\rk^{-1}$. The first and
second order parts of $\mathcal {J}^{tor}$ are given by
\[
\begin{split}
&\left( \mathcal {J}^{tor} \right)^{(1)} + \left( \mathcal {J}^{tor}
\right)^{(2)}\\
=& -2\sqrt{-1} M_p A_1(t)^{T} + 2\sqrt{-1} M_p \lk A_1(t)^{T} \rk^2 - 2\sqrt{-1} M_p A_2(t)^{T} \\
=& \sum_{i=1}^{n} t_i \int_{X_p} \theta_p^{\alpha} \wedge \mathbb{H}
( \mu_i \lrcorner \theta_p^{\beta} )
- \sum_{i,j=1}^{n} \frac{\sqrt{-1}}{2} t_i t_j \int_{X_p} \theta_p^{\alpha} \wedge \mathbb{H}
( \mu_i \lrcorner \theta_p^{\delta} ) M_p^{\delta\gamma} \int_{X_p} \theta_p^{\gamma} \wedge \mathbb{H} ( \mu_j \lrcorner \theta_p^{\beta})\\
& + \sum_{i,j=1}^{n} t_i t_j \int_{X_p} \theta^{\alpha}_p \wedge
\mathbb{H} ( \mu_i \lrcorner \eta^{\beta}_j ),
\end{split}
\]
where $\eta^{\alpha}_i = - \mathbb{G} \overline{\partial}^{\ast}
\partial ( \mu_i \lrcorner \theta_p^{\alpha})$ and
$M_p^{\alpha\beta}$ is the inverse matrix of $M_{p,\alpha\beta}$.
From the choice of $t_i$ above, for any $1\leq \alpha,\beta \leq g$, we have
\[ \int_{X_p} \theta_p^{\alpha} \wedge \mathbb{H} ( \mu_i \lrcorner \theta_p^{\beta} ) = 0, \]
where $2g \leq i \leq 3g-3$. Hence we can write out the image of $\{
D_{k} , D_{ij} \}_{1 \leq k \leq 2g-1 , 2g \leq i \leq j \leq 3g-3
}$ under $\mathcal {J}^{tor}_{\ast}$ by using the expansion formula
of $\mathcal {J}^{tor}$:
\[\begin{cases}
\mathcal {J}^{tor}_{\ast} ( D_k )= \int_{X_p} \theta_p^{\alpha} \wedge \mathbb {H} (\mu_k \lrcorner \theta_p^{\beta} ),     &    1 \leq k \leq 2g-1, \\
\mathcal {J}^{tor}_{\ast} ( D_{ij} )= \int_{X_p} \theta_p^{\alpha} \wedge \mathbb {H} \lk \mu_i \lrcorner \partial \mathbb{G} \overline{\partial}^{\ast}
(\mu_j \lrcorner \theta^{\beta}) \rk,     &    2g \leq i = j \leq 3g-3, \\
\mathcal {J}^{tor}_{\ast} ( D_{ij} )= \int_{X_p} \theta_p^{\alpha}
\wedge \mathbb {H} \lk \mu_i \lrcorner \partial \mathbb{G}
\overline{\partial}^{\ast} (\mu_j \lrcorner \theta^{\beta})  + \mu_j
\lrcorner \partial \mathbb{G} \overline{\partial}^{\ast} (\mu_i
\lrcorner \theta^{\beta}) \rk, &    2g \leq i < j \leq 3g-3.
\end{cases} \]
Finally we need to show that $\{ \mathcal {J}^{tor}_{\ast} ( D_k ) ,
\mathcal {J}^{tor}_{\ast} ( D_{ij} ) \}$ are linearly independent.
Since $X_p$ is a hyperelliptic Riemann Surface, these \v{C}ech
cohomology groups, such as $H^0 (\Omega^1_{X_p})$, $\check{H}^1
(\mathcal {O}_{X_p})$ and $\check{H}^1 (T_{X_p})$, have explicit
bases just as described in \cite{Kar,OS}. Moreover, these papers
have showed that these directions are linearly independent in terms
of \v{C}ech cohomology. We give a proof in Appendix \ref{app} that
our directions are actually the same as theirs, which completes the
proof of this theorem.
\end{proof}

Local Torelli Theorems \ref{lTT1} and \ref{lTT2} tell us that the
period map gives a local embedding of the Kuranishi coordinate chart
$\Delta_{p,\epsilon}$ when $p$ lies in the nonhyperelliptic locus,
and of $\Delta_{p,\epsilon} / \mathbb{Z}_2$ when $p$ lies in the
hyperelliptic locus. This local embedding induces a matrix model for
the local Kuranishi coordinates.

\begin{defi}
Matrix Model for the Kuranishi coordinate charts.

The image of the Kuranishi coordinate chart under the period map is
called the matrix model when the local Torelli theorems hold. Here
we identify the Kuranishi coordinate chart with its matrix model,
which lies in $\mathcal {H}_g \subset \mathbb {C}^{ \frac{g(g+1)}{2}
}$.
\end{defi}

\section{Proof of the global Torelli Theorems}\label{gTT}
This section is devoted to the proof of the global Torelli theorem for
Riemann Surfaces.

We recall some basic facts of Riemann Surface $S$ and its Jacobian $\mathrm{Jac}S$.
Fix one point $p_0$ on $S$. There is a natural map
\[ \begin{array}{cccc}
\chi_{p_0}: & S^{(d)} & \longrightarrow & \mathrm{Jac}S \\
& D = p_1 + \cdots + p_d & \longrightarrow &
\left[ \left( \sum_{i=1}^d \int_{p_0}^{p_i} \theta^1, \cdots, \sum_{i=1}^d \int_{p_0}^{p_i} \theta^{g} \right) \right], \\
\end{array} \]
where $S^{(d)} = \underbrace{S \times \cdots \times S}_{d\ \text{times}} / S_d$ with $S_d$ the symmetric group of order $d$
and $\{ \theta^{\alpha} \}_{\alpha =1}^g$ is the basis of $H^0(S, K_S)$.
\begin{defi}\label{IAM}
Define $W^r_d$ in $\mathrm{Jac}S$ associated to $S$:
\[ W^r_d := \left\{ \chi_{p_0} (D) \in \mathrm{Jac}S  \big| D \geq 0,\ \deg D = d\ \textrm{and}\ \dim |D| \geq r \right\} \]
and set $W_d$ as $W^0_d$.
\end{defi}

The polarization class determines the line bundles associated to it
up to translations on abelian varieties. Theta functions are holomorphic sections of those line bundles.
On principally polarized abelian varieties $(A,[\omega])$, $\dim H^0(A,L)=1$ where $L$ is associated to
the polarization class $[\omega]$. Theta divisor $\Theta$ is the zero locus of the generator section of $L$.
Thus the polarization class $[\omega]$ determines theta divisor up to translations.
Moreover, we have the following famous theorem on Jacobians.
\begin{theo}{ \emph{(Riemann's Theorem. See P. 338 of \cite{GH})} }\label{RT}
$\Theta = W_{g-1} - \frac{K}{2}$, where $K$ is the image of canonical divisor under $\chi_{p_0}$.
\end{theo}
Riemann's Theorem tells us
that the polarization class on $\mathrm{Jac}S$ determines $W_{g-1}$, which reflects the complex structure on $S$
to some extent, up to translations.

Also the intersection number $^{\#}\!( W_1 \cdot \Theta) = g$ and we have another theorem.
\begin{theo}{ \emph{(See P. 336 --- P. 339 of \cite{GH})} }\label{InterT1}
For $\lambda \in \mathrm{Jac}S$ such that $W_1 \nsubseteq \{ \Theta + \lambda \}$,
$W_1$ and $\Theta + \lambda$ have $g$ intersection points,
denoted by $p_1(\lambda), \cdots, p_g(\lambda)$. And the equality
\[ \sum_{i=1}^g p_i(\lambda) = \lambda + \frac{K}{2} \]
holds. Moreover, $W_1 \subseteq \{ \Theta + \lambda \} $ if and only if $\lambda + \frac{K}{2} \in W^1_g$.
Equivalently, when $W_1 \nsubseteq \{ W_{g-1} + b \}$, there are $p_1(b), \cdots, p_{g}(b)$
in $W_1 \bigcap \{ W_{g-1} + b \}$ and $\sum_{i=1}^g p_i(b) = b + K$.
Meanwhile, $W_1 \subseteq \{ W_{g-1} + b \}$ if only if $b + K \in W^1_g$.
\end{theo}
Define an operation of two sets $A, B$ in $\mathrm{Jac}S$:
\[ A \ominus B := \bigcap_{ x \in B} \{ A - x \}. \]
\begin{prop}\label{minus1}
For $0 \leq r \leq t \leq g-1$ and $a,b \in \mathrm{Jac}S$, $\{ W_t + a \} \ominus \{ W_r + b\} = W_{t-r} + a - b $.
\end{prop}

\begin{coro}\label{minus2}
For $0 \leq r \leq g-1, $
\begin{enumerate}
\item $ \{ W_{g-1} + a \} \ominus \{ W_r + b \} = W_{g-1-r} + a - b$,
\item $\{ W_{g-1} + a \} \ominus \{ - W_{r} + b \} = -W_{g-1-r} + a - b + K$.
\end{enumerate}
\end{coro}

These two proofs can be found on P. 155, P. 156, and P. 161 of \cite{FK}. And we give a
sketch of proof to the following proposition.

\begin{prop}\label{InterT2}
For $0 \leq r \leq g-2$, fix $a \in \mathrm{Jac}S$ and let $b = a + x - y$ where $x \in W_1$
and $y \in W_{g-1-r}$. Then either $ \{ W_{r+1} + a \} \subseteq \{ W_{g-1} + b \}$ or
\[ \{ W_{r+1} + a \} \bigcap \{ W_{g-1} + b \} = \{ W_r + a + x \} \bigcup T, \]
where $T = \{ W_{r+1} + a \} \bigcap \{ - W_{g-2} - y + a + K \}$.
\end{prop}

\begin{proof}[Sketch of proof]
$W_{r+1} + a$ and $W_{g-1} + b$ are two irreducible subvarieties of $\mathrm {Jac} S$.
If one is not contained in another, they will have intersection.
Thus we assume $\{ W_{r+1} + a \} \nsubseteq \{ W_{g-1} + b \}$.

Let $x = \chi_{p_0}(F)$
and $y = \chi_{p_0}(D)$ with effective divisors $F$ and $D$
of degree $1$ and $g-1-r$ respectively. $D$ can't contain the point $F$, otherwise
\[ W_{r+1} + a = W_{r+1} + b + y - x = W_{r+1} + b + \chi_{p_0}(F + D') - \chi_{p_0}(F)
= W_{r+1} + b + \chi_{p_0}(D') \subseteq W_{g-1} + b \]
Let $u \in \{ W_{r+1} + a \} \bigcap \{ W_{g-1} + b \} $.
Then there are two effective divisors $P$ and $Q$ of degree $r+1$ and $g-1$ respectively,
such that $u = \chi_{p_0} (P) + a = \chi_{p_0} (Q) + b$.
This implies $\chi_{p_0}(P + D) = \chi_{p_0}(Q + F)$. By Abel Theorem, $P + D \equiv Q + F$.
\begin{enumerate}
\item $\dim |P+D| = 0$. Then $P + D = Q + F$, implying that $F$ is contained in $P$.
Hence $u = \chi_{p_0}(P) + a = \chi_{p_0}(P' + F) + a = \chi_{p_0}(P') + x + a
\subseteq W_r + x + a$.
\item $\dim| P + D | > 0$. For any $\tilde{F} \in S$, there exists an effective divisor $\tilde{Q}$
such that $ P + D \equiv \tilde{F} + \tilde{Q}$.
Then $u = \chi_{p_0}( Q ) + b = \chi_{p_0}( \tilde{F} ) +  \chi_{p_0}( \tilde{Q} ) -  \chi_{p_0}(F) + b
= \chi_{p_0}( \tilde{F} ) +  \chi_{p_0}( \tilde{Q} ) + a - y$
implying $u \in \bigcap_{\tilde{F} \in S } \{ W_{g-1} +  \chi_{p_0}( \tilde{F} ) \} + a - y
= \{ W_{g-1} \ominus  - W_1 \} + a - y = - W_{g-2} + K + a - y $.
\end{enumerate}
Hence $\{ W_{r+1} + a \} \bigcap \{ W_{g-1} + b \} \subseteq \{ W_r + a + x \} \bigcup T$.
The reverse inclusion is clear.
\end{proof}

\begin{theo}\label{gTT1}
$\mathcal {J}^{tor}: \, \mathcal {T}or_g / \mathbb {Z}_2 \rightarrow \mathcal {H}_g$
is an embedding for $g\geq3$.
\end{theo}

\begin{proof}
From the discussion of Section \ref{lTT}, $\mathcal {T}or_g /
\mathbb{Z}_2$ is a complex orbifold of complex dimension $3g-3$. For
every point $p$ in the non-hyperelliptic locus, we have the
Kuranishi coordinate chart $\Delta_{p,\epsilon}$ centered at $p$,
which descends from $\mathcal {T}or_g$. As to the hyperelliptic
locus, we denote by $\Delta_{p,\epsilon} / \mathbb {Z}_2$ the local
coordinate chart around the hyperelliptic point according to the
local behavior of the hyperelliptic locus. From the local Torelli
theorems \ref{lTT1} and \ref{lTT2}, $\mathcal {J}^{tor}$ gives a
local embedding on both of these two kinds of coordinate charts. All
we need to show is that $\mathcal {J}^{tor}$ is injective. It is
easy to see that $\mathcal {T}or_g / \mathbb {Z}_2 \cong \mathcal
{T}_g / \widetilde {T}_g$. Thus the proof of one-to-one
correspondence between $\widetilde {T}_g$ orbit and its Jacobian is
our ultimate, which is equivalent to say that two points in
$\mathcal {T}_g$ with the same Jacobian must be related by some
element in $\widetilde {T}_g$.

According to \cite{RH1} and \cite{HL}, $\mathcal {H}_g$ can be
viewed as the isomorphism classes of principally polarized abelian
varieties together with a symplectic basis $(A, \gamma)$, where
$\gamma:\, H_1(\Sigma, \mathbb{Z}) \rightarrow H_1(A, \mathbb{Z})$
preserves the intersection paring on $\Sigma$ and the principally
polarized form on $A$. And the identification is given from
$(A, \gamma)$ to its period matrix with respect to this symplectic
basis. By changing the symplectic basis, we have the natural
$\mathrm {Sp} ( g, \mathbb{Z} )$ action on $\mathcal {H}_g$.
However, the kernel of the $\mathrm {Sp} ( g , \mathbb {Z} )$ action
is $\pm \ii_{2g}$. That is to say, for every principally polarized
abelian variety $A$,
\[ (A, \gamma) \cong (A, -\gamma). \]
Also $\mathcal {T}or_g$ can be identified with the isomorphism
classes of Riemann Surfaces together with a symplectic basis $(C,
\gamma)$, where $\gamma: \, H_1(\Sigma,\mathbb{Z}) \rightarrow
H_1(C,\mathbb{Z})$ preserves the intersection paring on $\Sigma$ and
$C$, since $\Gamma_g / T_g = \mathrm {Sp} (g,\mathbb{Z})$. Moreover,
the period map $\mathcal {J}^{tor}: \, \mathcal {T}or_g \rightarrow
\mathcal {H}_g$ is given by
\[ \begin{array}{rcl}
\mathcal {T}or_g & \longrightarrow & \mathcal {H}_g \\
(C,\gamma) & \longrightarrow & (\mathrm{Jac} C, \gamma),
\end{array} \]
where we have the natural isomorphism $H_1 ( C,\mathbb{Z} ) \cong H_1 (\mathrm{Jac}C,\mathbb{Z})$.

Now assume that two points $[C,[f]]$ and $[C',[f']]$ on $\mathcal
{T}_g$ are mapped to the same Jacobian, namely $(A, \gamma)$. Write
$(C,\gamma)$ and $(C', \gamma')$ on $\mathcal {T}or_g$ as the
corresponding two points descended from $[C,[f]]$ and $[C',[f']]$,
respectively. As $(C,\gamma)$ and $(C', \gamma')$ are mapped to the
same Jacobian $(A,  \gamma)$, their symplectic bases will be the
same up to a change of the sign. Without loss of generality, we may
assume that $(C, \gamma)$ and $(C', \gamma')$ share the same
symplectic basis after changing the sign.
Going back to the two
corresponding points on $\mathcal {T}_g$, the following picture
appears since we can see $\mathcal {T}_g$ from the deformation theoretic point of view:
\[ \xymatrix{ \Sigma & \Sigma \\
C \ar[u]^{f} \ar[r]^{\phi}& C' \ar[u]^{f'} },\] where $\phi$ is a
diffeomorphism obtained from the deformation of the complex
structures between $C$ and $C'$ with $[f \phi^{-1}
f'^{-1}] \in T_g$, since $C$ and $C'$ share the same symplectic basis.
Denote $F$ by $f \phi^{-1} f'^{-1}$. Then the commutative diagram follows
\[ \xymatrix{ & \Sigma & \\
C \ar[ur]^{f} \ar[rr]^{\phi}& & C'\ar[ul]_{Ff'} \\
}.\] In fact we will prove that $\phi$ is biholomorphic.
Hence $[C,[f]]$ and $[C',[f']]$ are related by $\widetilde {T}_g$.

To see this, we first recall the definition of the
Jacobian. The Jacobian of a Riemann Surface $X$ is nothing but
$\mathbb{C}^g \big / \Lambda$ and
\[ \Lambda = \mathbb{Z} \left\{ \int\limits_{A_1} \begin{pmatrix} \theta^1 \\ \vdots \\ \theta^g \\ \end{pmatrix},
\cdots, \int\limits_{A_g} \begin{pmatrix} \theta^1 \\ \vdots \\
\theta^g \\ \end{pmatrix}, \int\limits_{B_1} \begin{pmatrix}
\theta^1 \\ \vdots \\ \theta^g \\ \end{pmatrix}, \cdots,
\int\limits_{B_g} \begin{pmatrix} \theta^1 \\ \vdots \\ \theta^g \\
\end{pmatrix} \right\}, \] where $\{ \theta^{\alpha} \}_{\alpha =
1}^g$ is a basis of $H^0 (X, K_X)$. Since $[C,[f]]$ and $[C',[f']]$ are
mapped to the same Jacobian by the period map, their symplectic
bases $(A_{\alpha},B_{\alpha})$ and $(A'_{\alpha}, B'_{\alpha})$ are
related by $\phi$ together with
\begin{equation}\label{SameJ}
\int\limits_{A_{\alpha}} \begin{pmatrix} \theta^1 \\ \vdots \\
\theta^g \end{pmatrix} = \int\limits_{A'_{\alpha}} \begin{pmatrix}
\theta'^1 \\ \vdots \\ \theta'^g \end{pmatrix},\
\int\limits_{B_{\alpha}} \begin{pmatrix} \theta^1 \\ \vdots \\
\theta^g \end{pmatrix} = \int\limits_{B'_{\alpha}} \begin{pmatrix}
\theta'^1 \\ \vdots \\ \theta'^g \end{pmatrix}.
\end{equation}
Let
\[ \Omega = \begin{pmatrix} \theta^1 \\ \vdots \\ \theta^g \end{pmatrix},\
\Omega' = \begin{pmatrix} \theta'^1 \\ \vdots \\ \theta'^g
\end{pmatrix} .\] Set $\phi^{\ast} [ \theta'^{\alpha} ] = \sum_{\beta =
1}^{g} x_{\alpha \beta} [ \theta^{\beta} ] + \sum_{\beta = 1}^g
y_{\alpha\beta} [ \bar{\theta}^{\beta} ] $ where
$\phi^{\ast}: H^1_{ \mathrm{dR} } ( C', \mathbb{C} ) \rightarrow H^1_{ \mathrm{dR} }(C , \mathbb{C}) $
and $[\theta^{\beta}]$ is the de Rham class represented by $\theta^{\beta}$. Thus
\begin{equation}\label{ttheta}
\phi^{\ast} [ \Omega' ] =
\begin{pmatrix}  X & Y \end{pmatrix}
\begin{pmatrix} [ \Omega ] \\ [ \bar{\Omega} ] \end{pmatrix}.
\end{equation}
Put \eqref{SameJ} and \eqref{ttheta} together to get
\[ \int\limits_{A_{\alpha}} [ \Omega ]
= \int\limits_{A'_{\alpha}} [ \Omega'] = \int\limits_{\phi_{\ast}
A_{\alpha}} [ \Omega'] = \int\limits_{A_{\alpha}} \phi^{\ast} [ \Omega' ]=
\begin{pmatrix} X & Y \end{pmatrix} \int_{A_{\alpha}}
\begin{pmatrix} [ \Omega ] \\ [ \bar{\Omega} ]\\ \end{pmatrix}, \]
\[ \int\limits_{B_{\alpha}} [ \Omega ]
= \int\limits_{B'_{\alpha}} [ \Omega' ] = \int\limits_{\phi_{\ast}
B_{\alpha}} [ \Omega' ]= \int\limits_{B_{\alpha}} \phi^{\ast} [ \Omega' ] =
\begin{pmatrix} X & Y \end{pmatrix} \int_{B_{\alpha}}
\begin{pmatrix} [ \Omega ]\\ [ \bar{\Omega} ] \\ \end{pmatrix} .\]
Reformulating these two equalities into matrix form, we get
\begin{equation}\label{Periods}
\begin{pmatrix}
\int\limits_{A_{\alpha}} [ \Omega ] & \int\limits_{B_{\alpha}} [ \Omega ] \\
\int\limits_{A_{\alpha}} [ \bar{\Omega} ] & \int\limits_{B_{\alpha}}  [ \bar{\Omega} ] \\
\end{pmatrix}
= \begin{pmatrix} X & Y \\ \bar{Y} & \bar{X} \\ \end{pmatrix}
\begin{pmatrix} \int\limits_{A_{\alpha}} \begin{pmatrix} [ \Omega ] \\ [ \bar{\Omega} ] \end{pmatrix} &
\int\limits_{B_{\alpha}} \begin{pmatrix} [ \Omega ] \\ [ \bar{\Omega} ]
\end{pmatrix} \end{pmatrix}.
\end{equation}
Observe that $\det \begin{pmatrix}
\int\limits_{A_{\alpha}} [ \Omega ] & \int\limits_{B_{\alpha}} [ \Omega ] \\
\int\limits_{A_{\alpha}} [ \bar{\Omega} ] & \int\limits_{B_{\alpha}} [ \bar{\Omega} ] \\
\end{pmatrix} \neq 0$, since $\begin{pmatrix} [ \Omega ] \\ [ \bar{\Omega} ] \end{pmatrix}$ is the basis of
$H^1_{\mathrm{dR}} (C, \mathbb{C})$ and $( A_{\alpha}, B_{\alpha} )$ is the one of $H_1 ( C,\mathbb{Z} )$.
These imply that
\begin{equation}\label{relation}
X = \ii_g,\ Y = 0.
\end{equation}
Hence $\phi^{\ast}(\theta'^{\alpha}) = \theta^{\alpha} + df^{\alpha}$ for some $f^{\alpha}$.

Now we consider
\[ \begin{array}{cccc}
\chi_{p_0}: & C & \longrightarrow & \mathrm{Jac}C  \\
& p & \longrightarrow & \left[ \left( \int_{p_0}^{p} \theta^1, \cdots, \int_{p_0}^{p} \theta^{g} \right) \right]
\end{array} \]
embeds Riemann Surface $C$ into its Jacobian. The following diagram shows that two images of $C$ and $C'$,
denoted by $W_1$ and $V_1$ respectively, are related by $\phi$:
\[ \xymatrix {
(C,p_0) \ar[d]^{\phi} \ar@^{(->}[dr]^{\chi_{p_0}} \\
(C',\phi(p_0)) \ar@_{(->}[r]^{ \chi_{\phi(p_0)} } &  \mathbb{C}^g / \Lambda    }. \]
More precisely, let one smooth curve $\tau$ on $C$ connect $p_0$ and $p$ with $\phi(\tau)$ connecting $\phi(p_0)$ and $\phi(p)$.
Then we have
\[ \begin{split}
& \left[ \left( \int_{\phi(p_0)}^{\phi(p)} \theta'^1, \cdots, \int_{\phi(p_0)}^{\phi(p)} \theta'^{g} \right) \right]
= \left[ \left( \int_{\phi(\tau)} \theta'^1, \cdots, \int_{\phi(\tau)} \theta'^{g} \right) \right] \\
= & \left[ \left( \int_{\tau} \phi^{\ast} (\theta'^1), \cdots, \int_{\tau} \phi^{\ast} (\theta'^{g}) \right) \right]
= \left[ \left( \int_{\tau} ( \theta^1 + df^1 ), \cdots, \int_{\tau} ( \theta^{g} + df^g ) \right) \right] \\
= & \left[ \left(  f^1(p) - f^1(p_0) + \int_{\tau} \theta^1, \cdots, f^g(p) - f^g(p_0) + \int_{\tau} \theta^{g} \right) \right]
= \left[ \left( f^1(p) + \int_{p_0}^{p} \theta^1, \cdots, f^g(p) + \int_{p_0}^{p} \theta^{g} \right) \right]  \\
= & \left[ \bigg( f^1(p), \cdots, f^g(p) \bigg)
+ \left( \int_{p_0}^{p} \theta^1, \cdots, \int_{p_0}^{p} \theta^{g} \right) \right].  \\
\end{split} \]
Hence $W_1$ and $V_1$ are different by a varying vector $\left[ \left( f^1(p), \cdots, f^g(p) \right) \right]$
and here we normalize $f^i,\ 1 \leq i \leq g$ such that $f^i(p_0) = 0$. We would like to use the same polarization
(actually the same theta divisor) in the Jacobian to show that varying vector to be constant. Afterwards we will
associate $W_d^r$ and $V_d^r$ to $C$ and $C'$ through the mappings $\chi_{p_0}$ and $\chi_{\phi(p_0)}$ respectively
just as Definition \ref{IAM}.

Consider the smallest integer $r$ such that
\[ V_1 \subseteq W_{r+1} + a\     \mathrm{or}\     V_1 \subseteq -W_{r+1} + a \]
for some $a \in \mathbb{C}^g / \Lambda$. It is easy to see that $r \leq g-2$. Actually from Theorem \ref{RT},
\begin{equation}\label{VandW}
V_{g-1} - \frac{K'}{2} = W_{g-1} - \frac{K}{2}.
\end{equation}
Together with Theorem \ref{InterT1} and \eqref{VandW}, it will happen that $V_1 \subseteq W_{g-1} + b$ for $b + \frac{K+K'}{2} \in V^1_g$.

$\underline{\textrm{Case}\ 1: r=0}$. $V_1 = - W_1 + a $ or $V_1 = W_1 + a $.
But $V_1$ and $W_1$ start through the origin of $\mathbb{C}^g / \Lambda$.
Thus $a=0$. $V_1 = - W_1$ which means that $\phi_{\ast}$ reverses the symplectic basis of these two Riemann Surfaces.
Contradict with our assumption on $\phi$ ahead.
Thus $V_1 = W_1$, which forces all $f^i$ to be zero.

$\underline{\textrm{Case}\ 2: r>0}$. Suppose $V_1 \subseteq W_{r+1} + a$. Set $b = a + x - y$ where $x \in W_1$ and $y$ varies in $W_{g-1-r}$.
For fixed $x$, $V_1$ can't always lie in $W_{g-1} + b$ when $y$ runs through $W_{g-1-r}$. If so, then we would have
\[ V_1 \subseteq \bigcap_{y \in W_{g-1-r}} \{ W_{g-1} + a + x - y \} = \{ W_{g-1} \ominus W_{g-1-r} \} + a + x = W_r + a + x \]
by Corollary \ref{minus2}, contradicting to the minimality of $r$. Hence we have two following results
\begin{enumerate}
\item For any fixed $x \in W_1$, $V_1 \cap \{ W_{g-1} + b \}$ will be $g$ points for generic $y \in W_{g-1-r}$.
\item There exists some $y \in W_{g-1-r}$ such that $V_1 \cap \{ W_{g-1} + b \}$ will be $g$ points for generic $x \in W_1$.
Because it is impossible that for any fixed $y \in W_{g-1-r}$, $V_1 \subseteq \{ W_{g-1} + a + x - y \}$ when $x$ runs through $W_1$.
\end{enumerate}
Under the circumstance of the result $(1)$, we have, by Proposition \ref{InterT2},
\begin{equation}\label{InterF}
\begin{split}
V_1 \bigcap \{ W_{g-1} + b \} & = V_1 \bigcap \{ W_{g-1} + b \} \bigcap \{ W_{r+1} + a \} \\
                              & = \left( V_1 \bigcap \{ W_{r} + a + x \} \right) \bigcup \left( V_1 \bigcap T \right),   \\
\end{split}
\end{equation}
where $ V_1 \bigcap \{ W_r + a + x \}$ depends on $x$, while $V_1 \bigcap T$ on $y$.
Write the $g$ intersection points as $p_1(b), \cdots, p_g(b)$. From \eqref{VandW} and Theorem \ref{InterT1}, we have
\begin{equation}\label{Intersection}
\sum_{i=1}^g p_i(b) = b + \frac{K+K'}{2} = a + x - y + \frac{K+K'}{2}.
\end{equation}

\textbf{Claim}: For any fixed $x \in W_1$, $V_1 \bigcap \{ W_r + a + x \}$ has at most one point.

In fact, if there are two points in $V_1 \bigcap \{ W_r + a + x \}$ for some $x \in W_1$,
fixing that $x$, we know that equality \eqref{Intersection} holds for generic $y \in W_{g-1-r}$, leaving $p_1(b)$ and $p_2(b)$ fixed,
which implies that $a + x - W_{g-1-r} \subseteq \{ V_{g-2} + c \}$ with $c$ a constant.
Hence $-W_{g-1-r} \subseteq \{ V_{g-2} + c' \}$ with $c'$ a constant. By Corollary \ref{minus2}, we get
\[ V_1 = V_{g-1} \ominus V_{g-2} \subseteq \{ W_{g-1} + \frac{K'-K}{2}\} \ominus \{ -W_{g-1-r} - c'\}
= - W_r + \frac{K+K'}{2} + c', \]
contradicting to the minimality of $r$. Thus our claim is proved.

As $V_1$ and $W_{r} + x + a$ are subvarieties in $W_{r+1} + a$ with complementary dimensions,
$^{\#}\!(V_1 \cdot \{ W_r + x + a \}) \leq 1$ for all $x \in W_1$ from our claim.
And $W_r + x + a$ and $W_r + a$ have the same homology class. In fact, let us denote the origin of the Jacobian by $x_0$.
Consider a $C^{\infty}$ curve $\gamma(t),\ t \in [0,1]$ between $x_0$ and $x$ on $W_1$
such that $\gamma(0) = x_0$ and $\gamma(1) = x$. Then we have
$\partial ( \bigcup_{y \in \gamma(t)} \{ W_r + y +a \} ) = \{ W_r + x + a \} - \{ W_r + a \}$.
Hence $^{\#}\!(V_1 \cdot \{ W_r + x + a \}) = ^{\#}\!(V_1 \cdot \{ W_r + a \})$.
But the constant $^{\#}\!(V_1 \cdot \{ W_r + x + a \}) $ for all $x \in W_1$ can't be zero
since $V_1 = V_1 \bigcap \{ W_{r+1} + a\} = \bigcup_{ x \in W_1} \left( V_1 \bigcap \{W_r + x + a\} \right)$.
Thus $^{\#}\!(V_1 \cdot \{ W_r + x + a \}) = 1$ for all $x \in W_1$,
namely $V_1 \bigcap \{ W_r + x +a\}$ has one point for all $x \in W_1$.

Apply the result $(2)$. Fix that $y$ and we still have \eqref{InterF}. Equality \eqref{Intersection} holds for generic $x \in W_1$, leaving $p_2(b), \cdots, p_g(b)$ fixed.
Hence $V_1 + c'' \subseteq a + W_1 -y + \frac{K+K'}{2}$ with $c''$ a constant. This contradicts to the minimality of $r$. Therefore
the case $r > 0$ is impossible.

Now we have proved that $f^i \equiv 0,\ 1 \leq i \leq g$. Hence $\phi^{\ast}$ preserves holomorphic one forms from $C'$
to $C$. Choose coordinates centered at $p$ and $\phi(p)$, which are denoted by $(z,p)$ and $(w,\phi(p))$.
Pick $\Xi \in H^0(C',K_{C'})$ with $\Xi(\phi(p)) \neq 0$. Locally $\Xi$ can be written as
\[ \Xi = g(w) dw. \]
Pull $\Xi$ back by $\phi$, then we get holomorphic one form on $C$. However,
\[ \phi^{\ast} \Xi  = g(\phi(z)) \frac {\partial \phi} {\partial z} dz
+ g(\phi(z)) \frac {\partial \phi} {\partial \bar{z}} d \bar{z}. \]
Then $g(\phi(z)) \frac {\partial \phi} {\partial \bar{z}} = 0$. Restricting to the point $p$ and $g(\phi(p)) \neq 0$,
we have $\left. \frac {\partial \phi} {\partial \bar{z}} \right|_{z=0} = 0$. At last, $\phi$ is holomorphic,
finishing the proof of the theorem.
\end{proof}

\begin{rema}
In the discussion of the case $r>0$, if we suppose $V_1 \subseteq - W_{r+1} - a$, take $- W_{g-1} - b$ to intersect with
$V_1$. Following the same method, we will get $V_1 \subseteq - W_1 + c$ with $c$ a constant, contradicting the minimality
of $r$.
\end{rema}

\begin{coro}\label{gequal2}
For the case of $g=2$, $\mathcal {J}^{tor}:\, \mathcal {T}or_g
\rightarrow \mathcal {H}_g$ is an open embedding.
\end{coro}

\begin{proof}
Theorem \ref{lTT1} tells us that $\mathcal{J}^{tor}:\ \mathcal {T}or_g \rightarrow \mathcal {H}_g$
is an immersion everywhere when $g=2$. Besides, $\mathcal {J}^{tor}$ is an open map from
the fact $\dim_{\mathbb{C}} \mathcal {T}or_g = \dim_{\mathbb{C}}
 \mathcal {H}_g = 3$.  Moreover, Proposition \ref{Z2action} implies that
$\mathbb{Z}_2$ is a trivial action on $\mathcal {T}or_g$ since any
Riemann Surface with $g=2$ is hyperelliptic, indicating that
$\widetilde{T}_g$ orbit is the same as $T_g$ orbit on $\mathcal
{T}_g$. The proof of Theorem \ref{gTT1} implies that $\mathcal
{J}^{tor}$ is an open embedding.
\end{proof}

\begin{coro}\label{gTT2}
$\mathcal {J}^{tor} : \mathcal {T}or_g \rightarrow \mathcal {H}_g$ is a $2:1$ branched covering map branched over
$\mathcal {H}\mathcal {E}\mathcal {T}or_g$ onto its image for $g\geq3$.
\end{coro}

\begin{proof}
This is a direct consequence of Theorem \ref{gTT1}.
\end{proof}

\begin{prop}\label{orbit}
Let $\Delta_{p,\epsilon}$ be the Kuranishi coordinate chart on
$\mathcal {T}_g$. The period map $\Pi$ maps the $\Gamma_g$ orbit of
$\Delta_{p,\epsilon}$ onto the $\mathrm{Sp}(g,\mathbb{Z})$ orbit of
its image in $\mathcal {H}_g$.
\end{prop}

\begin{proof}
Recall that the Kuranishi coordinate chart $\Delta_{p,\epsilon}$ is given by
\[ \begin{array}{ccc}
\Delta_{p,\epsilon} & \rightarrow & \mathcal{T} _g \\
t & \rightarrow & [X_t,[F_t]],
\end{array} \]
where $(\mathcal {X},F)$ is the Kuranishi family with the Teichm\"{u}ller structure of $(X_p,[F_0])$ over
$\Delta_{p,\epsilon}$, while the coordinate map of $\Delta_{p,\epsilon}^{[\phi]}$ can be written as
\[ \begin{array}{ccc}
\Delta_{p,\epsilon}^{[\phi]} & \rightarrow & \mathcal{T} _g \\
t & \rightarrow & [X_t,[\phi F_t]],
\end{array} \]
where $\Delta_{p,\epsilon}^{ [\phi] } := [\phi]
\Delta_{p,\epsilon}$. Now the Kuranishi family becomes $(\mathcal
{X}, ( \phi \times 1 ) F)$, where $\phi \times 1:\,\Sigma \times
\Delta_{p,\epsilon}^{[\phi]} \rightarrow \Sigma \times
\Delta_{p,\epsilon}^{[\phi]}$, the same family as $(\mathcal {X},
F)$ up to a different Teichm\"{u}ller structure. Two symplectic bases
are linked by $\rho([\phi])$, denoted by $\begin{pmatrix} U & V
\\ R & S \\ \end{pmatrix} \in \mathrm{Sp}( g,\mathbb{Z} )$, i.e.,
\[ \begin{pmatrix} \widetilde{ A } \\ \widetilde{B} \\ \end{pmatrix} = \begin{pmatrix} U & V \\ R & S \\ \end{pmatrix}
\begin{pmatrix} A \\ B \\ \end{pmatrix},\] where
$\begin{pmatrix} \widetilde{ A } \\ \widetilde{B} \\
\end{pmatrix}$ and $\begin{pmatrix} A \\ B \\
\end{pmatrix}$ are the symplectic bases on
$\Delta_{p,\epsilon}^{[\phi]}$ and $\Delta_{p,\epsilon}$,
respectively. As we have seen, the matrix model of
$\Delta_{p,\epsilon}$ is
$$\Pi(t) = \begin{pmatrix}
\bar{\pi}_p & \pi_p \\ \ii_g  &  \ii_g \\
\end{pmatrix}\curvearrowright A(t)^{T}.$$ While on
$\Delta_{p,\epsilon}^{[\phi]}$, one has
\[ \begin{split}
\widetilde{\Pi}(t)_{\alpha\beta} &= \int_{ \widetilde{B}_{\alpha} } \widetilde{\sigma}(t)^{\gamma\beta} \theta^{\gamma}_p (t)\\
& = \int_{ \widetilde{B}_{\alpha} } \widetilde{\sigma}(t)^{\gamma\beta} \lk \theta^{\gamma}_p + A(t)^{\gamma}_{\delta} \bar{\theta}_p^{\delta} \rk\\
& = \int\limits_{ R_{\alpha\lambda} A_{\lambda} + S_{\alpha\lambda}
B_{\lambda} } \widetilde{\sigma}(t)^{\gamma\beta} \lk
\theta^{\gamma}_p
+ A(t)^{\gamma}_{\delta} \bar{\theta}_p^{\delta} \rk\\
& = \lk R_{\alpha\gamma} + S_{\alpha\lambda} \pi_p,_{\lambda\gamma}
+ R_{\alpha\delta} A(t)^{\gamma}_{\delta} + S_{\alpha\lambda}
\bar{\pi}_p,_{\lambda\delta} A(t)^{\gamma}_{\delta} \rk
\widetilde{\sigma}(t)^{\gamma\beta},
\end{split} \]
where $\widetilde{\sigma}(t)^{\alpha\beta}$ is the inverse matrix of $\widetilde{\sigma}(t)_{\alpha\beta}$. And
$\widetilde{\sigma}(t)_{\alpha\beta}$ is given by
\[  \begin{split}
\widetilde{\sigma}(t)_{\alpha\beta} & = \int_{ \widetilde{A}_{\alpha} } \theta^{\beta}_p(t) \\
& = \int\limits_{ U_{\alpha\lambda}A_{\lambda} + V_{\alpha\lambda}B_{\lambda} } \theta^{\beta}_p + A(t)^{\beta}_{\gamma} \bar{\theta}_p^{\gamma} \\
& = U_{\alpha\beta} + V_{\alpha\gamma} \pi_p,_{\gamma\beta} +
U_{\alpha\gamma} A(t)^{\beta}_{\gamma} + V_{\alpha\gamma}
\bar{\pi}_p,_{\lambda\gamma} A(t)^{\beta}_{\gamma} .\end{split} \]
Then we formulate all these into the matrix form:
\[ \begin{split}
\widetilde{\Pi}(t) & = \lk R (\ \ii_g +A(t)^{T} ) + S (\ \pi_p +
\bar{\pi}_p A(t)^{T} ) \rk \lk U (\ \ii_g +A(t)^{T} )
+ V (\ \pi_p + \bar{\pi}_p A(t)^{T} ) \rk^{-1} \\
& = \begin{pmatrix} S & R \\ V & U \\ \end{pmatrix}  \begin{pmatrix} \bar{\pi}_p & \pi_p \\ \ii_g & \ii_g \\ \end{pmatrix}\curvearrowright A(t)^{T} \\
& = \begin{pmatrix} S & R \\ V & U \\ \end{pmatrix}\curvearrowright
\Pi(t) .\end{split} \] Thus the $\Gamma_g$ orbit of
$\Delta_{p,\epsilon}$ is mapped, by the period map, onto the
$\mathrm{Sp}(g,\mathbb{Z})$ orbit of its matrix model $\Pi(t)$ in
$\mathcal {H}_g$, since the representation $\rho: \,\Gamma \rightarrow
\mathrm{Sp}(g,\mathbb{Z})$ is surjective.
\end{proof}

Denote by $\nu$ the transformation of $\mathrm {Sp}( g , \mathbb{Z} )$
\[ \begin{array}{ccc}
\mathrm {Sp} ( g, \mathbb{Z} ) & \rightarrow & \mathrm {Sp} ( g , \mathbb {Z} ) \\
\begin{pmatrix} U & V \\ R & S \\ \end{pmatrix} & \rightarrow & \begin{pmatrix} S & R \\ V & U  \\ \end{pmatrix} \\
\end{array} \]
and it is obvious that $\nu^2 = 1$.

\begin{theo}\emph{(Global Torelli Theorem on moduli space)}
$\mathcal {J}: \,\mathcal {M}_g \rightarrow \mathcal {A}_g$ is injective for $g\geq2$.
\end{theo}

\begin{proof}
As we have seen from Corollary \ref{gTT2}, $\mathcal {J}^{tor}:
\,\mathcal {T}or_g \rightarrow \mathcal {H}_g$ is a $2:1$ branched
covering map onto its image, branching over $\mathcal {H}\mathcal
{E}\mathcal {T}or_g$ for $g\geq3$. That is to say that the
$\widetilde {T}_g$ orbits on $\mathcal {T}_g$ have one-to-one
correspondence to their Jacobian given by the period map $\Pi$. This
is also true for $g=2$, from the proof of Corollary \ref{gequal2}.
From Proposition \ref{orbit}, the $\Gamma_g$ orbits are mapped onto
$\mathrm{Sp} ( g , \mathbb {Z} )$ orbits. Assume that two $\Gamma_g$
orbits $[p]$ and $[q]$ of $\mathcal {M}_g$ are mapped to the same
$\mathrm {Sp}(g,\mathbb{Z})$ orbit by $\mathcal {J}$. We lift these
to $\Pi:\, \mathcal {T}_g \rightarrow \mathcal {H}_g$ and thus have
$\Pi(p) = L \curvearrowright \Pi(q)$ for some $L \in \mathrm {Sp} (
g , \mathbb{Z})$. There is the following exact sequence
\[ 1 \rightarrow T_g \rightarrow \Gamma_g \stackrel{\rho}{\rightarrow} \mathrm{Sp} (g,\mathbb{Z}) \rightarrow 1. \]
Pick $[\phi] \in \rho^{-1} ( \nu (L) )$. Then $\Pi ( [\phi]q ) = L \curvearrowright \Pi(q)$ by Proposition \ref{orbit}.
Hence $p$ and $[\phi]q$ are in the same $\widetilde{T}_g$ orbit, which implies that $p$ and $q$ are in the same
$\Gamma_g$ orbit.
\end{proof}

\section{Appendix}\label{app}
Recall that the natural isomorphism between the \v{C}ech cohomology
$\check{H}^1 (T_X)$ and the Dolbeault cohomology
$H_{\overline{\partial}}^{0,1} (T_X)$, and isomorphism between
$\check{H}^1 (\mathcal {O}_X)$ and $H^{0,1}_{\overline{\partial}}$
follows similarly. Assume that there is an open covering
$\bigcup_{\alpha} U_{\alpha}$ on $X$ and then the natural
isomorphism $\Psi$ is given by
\[  \begin{array}{cccc}
\Psi: & \check{H}^1 (T_X) & \longrightarrow & H_{\overline{\partial}}^{0,1} (T_X) \\
 & [\theta_{\alpha\beta}]    & \longrightarrow  & [ \overline{\partial} \xi^{\alpha} ] \\
\end{array}, \]
where $\xi^{\alpha} \in A^{0,0}(U_{\alpha}, T_X)$ and $\xi^{\beta} - \xi^{\alpha} = \theta_{ \alpha \beta }$.

Now we return to the proof of the Theorem \ref{lTT2}, that is, $X$
is a hyperelliptic Riemann Surface, covered by two affine charts
$U_0$ and $U_1$ as described in \cite[P. 568]{Kar}. The first derivative
of the period map in the direction $D_k,1 \leq k \leq 2g-1$ in terms of
\v{C}ech cohomology is given by
\[ \begin{array}{ccc}
H^0(X,\Omega^1) & \longrightarrow & \check{H}^1 (X,\mathcal {O}_X)  \\
\omega & \longrightarrow & [\theta_{k} \lrcorner \omega] \\
\end{array}, \]
where $[\theta_k] \in \check{H}^1(T_X)$ corresponds to
$[\mu_k] \in H^{0,1}_{\overline{\partial}} (T_X)$. It is obvious that
$[\theta_k \lrcorner \omega]$ is mapped to $[\mu_k \lrcorner \omega]$
by the natural isomorphism from $\check{H}^1(\mathcal {O}_X)$ to
$H^{0,1}_{\overline{\partial}}$.
The second derivative of period map in the direction of $D_{ij},2g \leq i < j
\leq 3g-3$ in terms of \v{C}ech cohomology is given by
\[ \begin{array}{ccc}
H^0(X,\Omega^1) & \longrightarrow & \check{H}^1 (X,\mathcal {O}_X)  \\
\omega & \longrightarrow & [\theta_{j} \lrcorner \mathcal {L}_{\theta_i} \omega] \\
\end{array}, \]
where $\mathcal {L}_{\theta_i}$ denotes Lie derivative along
$\theta_i$ and $[\theta_i] \in \check{H}^1( T_X )$ corresponds to
$[\mu_i] \in H^{0,1}_{\overline{\partial}} (T_X)$. It is easy to see
that $\theta_{j} \lrcorner \mathcal {L}_{\theta_i} \omega =
\theta_{j} \lrcorner \partial ( {\theta_i} \lrcorner \omega )$.
Hence we need to show that $[\theta_{j} \lrcorner \partial (
{\theta_i} \lrcorner \omega )]$ is mapped to $[\mu_i \lrcorner
\partial \mathbb{G} \overline{\partial}^{\ast} ( \mu_j
\lrcorner \omega ) + \mu_j \lrcorner \partial \mathbb{G}
\overline{\partial}^{\ast} ( \mu_i \lrcorner \omega )]$ by the
natural isomorphism from $\check{H}^1 ( \mathcal {O}_X )$ to
$H^{0,1}_{\overline{\partial}}$. The $i=j$ case follows from almost
the same method as below. By the natural isomorphism between
$\check{H}^1 (X, T_X)$ and $H^{0,1}_{\overline{\partial}} (X,T_X)$,
we get $\xi_i^1\in A^{0,0}(U_1,T_X)$ and $\xi_i^0\in A^{0,0}
(U_0,T_X)$ such that
\begin{equation}\label{ISO}
\begin{cases}
\xi^1_i - \xi^0_i = \theta_i, \\
\overline{\partial} \xi^{1}_i = \mu_i + \overline{\partial} f_i,
\end{cases}
\end{equation}
where $f_i \in A^{0,0} ( X , T_X )$. As $\mu_i$ can change in the
Dolbeault cohomology class, we can assume that $\mu_i \big |_{U_0
\cap U_1} = 0$. In fact, $\mu_i$ is $\overline{\partial}$-closed and
thus locally $\overline{\partial}$-exact, i.e., $\mu_i =
\overline{\partial} h_i$ on $U_0 \cap U_1$. The desired
representative can be chosen as $\mu_i - \overline{\partial} ( \rho
h_i )$, where $\rho$ is the suitable cut-off function. Moreover we
can choose $f_i$ such that $f_i = \xi^0_i$ on $U_0 \cap U_1$, for
example $f_i := \rho\xi^0_i \big |_{U_0 \cap U_1}$. By use of
\eqref{ISO}, on $U_1$, we have
\begin{align*}
&\mu_i \lrcorner \partial \mathbb{G}
\overline{\partial}^{\ast}( \mu_j \lrcorner \omega )+ \mu_j \lrcorner \partial \mathbb{G} \overline{\partial}^{\ast} ( \mu_i \lrcorner \omega )\\
 =&\overline{\partial} (\xi_i^1 - f_i) \lrcorner \partial \Big( (\xi^1_j - f_j) \lrcorner \omega \Big)
+ \overline{\partial} (\xi_j^1 - f_j ) \lrcorner \partial \Big(
(\xi^1_i - f_i) \lrcorner \omega \Big).
\end{align*}
Similarly, we have an analogous equality on $U_0$.

Now we shall identify the \v{C}ech and Dolbeault cohomology classes
above. This question is equivalent to finding $\phi_{ij}^1$ and
$\phi_{ij}^0$ belonging to $A^{0,0}(U_1)$ and $A^{0,0}(U_0)$,
satisfying the following equations
\begin{equation}\label{OBS}
\begin{cases}
\overline{\partial} \phi_{ij}^1  =  \overline{\partial} (\xi_i^1 -
f_i) \lrcorner \partial \Big( (\xi^1_j - f_j) \lrcorner \omega \Big)
+  \overline{\partial} (\xi_j^1 - f_j ) \lrcorner \partial \Big( (\xi^1_i - f_i) \lrcorner \omega \Big),\\
\overline{\partial} \phi_{ij}^0  =  \overline{\partial} (\xi_i^0 -
f_i) \lrcorner \partial \Big( (\xi^0_j - f_j) \lrcorner \omega \Big)
+  \overline{\partial} (\xi_j^0 - f_j ) \lrcorner \partial \Big( (\xi^0_i - f_i) \lrcorner \omega \Big),\\
\phi^1_{ij} - \phi^0_{ij} = \theta_{j} \lrcorner \partial (
{\theta_i} \lrcorner \omega ).
\end{cases}
\end{equation}
It is obvious that the solutions of the first two equalities of
\eqref{OBS} always exist since the right hand sides of these two
equalities are $(0,1)$-forms and clearly
$\overline{\partial}$-closed. As
\begin{align*}
&\overline{\partial} (\xi_i^1 - f_i) \lrcorner \partial \Big(
(\xi^1_j - f_j) \lrcorner \omega \Big)
+ \overline{\partial} (\xi_j^1 - f_j ) \lrcorner \partial \Big( (\xi^1_i - f_i) \lrcorner \omega \Big) \\
=&\ \overline{\partial} \Big( (\xi_i^1 - f_i) \lrcorner \partial \lk
(\xi^1_j-f_j) \lrcorner \omega \rk \Big)
- (\xi_i^1 - f_i) \lrcorner \partial \Big( \overline{\partial} ( \xi^1_j - f_j ) \lrcorner \omega \Big) \\
&+ \overline{\partial} \Big( (\xi_j^1 - f_j) \lrcorner \partial \lk
(\xi^1_i-f_i) \lrcorner \omega \rk \Big) - (\xi_j^1 - f_j) \lrcorner
\partial \Big( \overline{\partial} ( \xi^1_i - f_i ) \lrcorner
\omega \Big),
\end{align*}
we can write $\phi^1_{ij}$ as
\begin{align*}
&(\xi_i^1 - f_i) \lrcorner \partial \Big( (\xi^1_j-f_j) \lrcorner
\omega \Big) + (\xi_j^1 - f_j)
\lrcorner \partial \Big( (\xi^1_i-f_i) \lrcorner \omega \Big)\\
&- \overline{\partial}^{-1} \Big( (\xi^1_i - f_i) \lrcorner \partial
\lk \overline{\partial} (\xi_j^1 -f_j ) \lrcorner \omega \rk \Big) -
\overline{\partial}^{-1} \Big( (\xi^1_j - f_j) \lrcorner
\partial \lk \overline{\partial} (\xi_i^1 -f_i ) \lrcorner \omega
\rk \Big),
\end{align*}
where $\overline{\partial}^{-1} \Big( (\xi^1_i - f_i) \lrcorner
\partial \lk \overline{\partial} (\xi_j^1 -f_j ) \lrcorner \omega
\rk \Big)$ stands for some solution $g$ satisfying
$$\overline{\partial} g =
(\xi^1_i - f_i) \lrcorner \partial \lk \overline{\partial} (\xi_j^1
-f_j ) \lrcorner \omega \rk.$$ This notation is reasonable as the
solution always exists. Thus
\begin{align*}
   &\phi^1_{ij} - \phi^0_{ij}\\
 =&(\xi_i^1 - f_i) \lrcorner \partial \Big( (\xi^1_j-f_j) \lrcorner \omega \Big)- (\xi_i^0 - f_i) \lrcorner \partial \Big( (\xi^0_j-f_j) \lrcorner
\omega \Big)
\\
&+ (\xi_j^1 - f_j) \lrcorner \partial \Big( (\xi^1_i-f_i) \lrcorner
\omega \Big)
- (\xi_j^0 - f_j) \lrcorner \partial \Big( (\xi^0_i-f_i) \lrcorner \omega \Big) \\
&- \overline{\partial}^{-1} \Big( (\xi^1_i - \xi^0_i) \lrcorner
\partial \lk \overline{\partial} (\xi_j^1 -f_j ) \lrcorner \omega
\rk \Big)
- \overline{\partial}^{-1} \Big( (\xi^1_j - \xi^0_j) \lrcorner \partial \lk \overline{\partial} (\xi_i^1 -f_i ) \lrcorner \omega \rk \Big) \\
 =&(\xi^1_i - \xi^0_i) \lrcorner \partial \Big( (\xi^1_j - f_j)
\lrcorner \omega \Big)
+ ( \xi^0_i - f_i) \lrcorner \partial \Big( (\xi^1_j - \xi^0_j)\lrcorner \omega \Big)\\
&+ ( \xi^1_j - \xi^0_j ) \lrcorner \partial \Big( (\xi^1_i - f_i)
\lrcorner \omega \Big)
+ (\xi^0_j - f_j) \lrcorner \partial \Big( (\xi^1_i - \xi^0_i) \lrcorner \omega \Big) \\
&- \overline{\partial}^{-1} \Big( (\xi^1_i - \xi^0_i) \lrcorner
\partial \lk \overline{\partial} (\xi_j^1 -f_j ) \lrcorner \omega
\rk \Big)
- \overline{\partial}^{-1} \Big( (\xi^1_j - \xi^0_j) \lrcorner \partial \lk \overline{\partial} (\xi_i^1 -f_i ) \lrcorner \omega \rk \Big) \\
 =&(\xi^1_i - \xi^0_i) \lrcorner \partial \Big( (\xi^1_j - f_j)
\lrcorner \omega \Big)+ ( \xi^0_i - f_i) \lrcorner \partial \Big( (\xi^1_j - \xi^0_j)\lrcorner \omega \Big) \\
&+(\xi^1_j - \xi^0_j) \lrcorner \partial \Big( (\xi^1_i -\xi^0_i)
\lrcorner \omega \Big)+ (\xi^1_j -\xi^0_j) \lrcorner \partial \Big(
(\xi^0_i - f_i) \lrcorner \omega \Big)
+ (\xi^0_j - f_j) \lrcorner \partial \Big( (\xi^1_i - \xi^0_i) \lrcorner \omega \Big) \\
&- \overline{\partial}^{-1} \Big( (\xi^1_i - \xi^0_i) \lrcorner
\partial \lk \overline{\partial} (\xi_j^1 -f_j ) \lrcorner \omega
\rk \Big)
- \overline{\partial}^{-1} \Big( (\xi^1_j - \xi^0_j) \lrcorner \partial \lk \overline{\partial} (\xi_i^1 -f_i ) \lrcorner \omega \rk \Big) \\
 =&(\xi^1_j - \xi^0_j) \lrcorner \partial \Big( (\xi^1_i -\xi^0_i) \lrcorner \omega \Big) \\
&+ (\xi^0_j - f_j) \lrcorner \partial \Big( (\xi^1_i - \xi^0_i)
\lrcorner \omega \Big)
+ ( \xi^0_i - f_i) \lrcorner \partial \Big( (\xi^1_j - \xi^0_j)\lrcorner \omega \Big) \\
&+ (\xi^1_j -\xi^0_j) \lrcorner \partial \Big( (\xi^0_i - f_i)
\lrcorner \omega \Big)- \overline{\partial}^{-1} \Big( (\xi^1_j -
\xi^0_j) \lrcorner \partial \lk \overline{\partial} (\xi_i^1 -f_i )
\lrcorner \omega \rk \Big)
\\
& + (\xi^1_i - \xi^0_i) \lrcorner \partial \Big( (\xi^1_j - f_j)
\lrcorner \omega \Big)- \overline{\partial}^{-1} \Big( (\xi^1_i -
\xi^0_i) \lrcorner \partial \lk \overline{\partial} (\xi_j^1 -f_j )
\lrcorner \omega \rk \Big)
 \\
=& \theta_j \lrcorner \partial ( \theta_i \lrcorner \omega )+ (
\xi^0_i - f_i) \lrcorner \partial \Big( (\xi^1_j - \xi^0_j)\lrcorner
\omega \Big)
+ (\xi^0_j - f_j) \lrcorner \partial \Big( (\xi^1_i - \xi^0_i) \lrcorner \omega \Big) \\
= &\theta_j \lrcorner \partial ( \theta_i \lrcorner \omega ).
\end{align*}
The penultimate equality results from
\[ \bar{\partial} \Bigg( (\xi^1_j -\xi^0_j) \lrcorner \partial \Big( (\xi^0_i - f_i)
\lrcorner \omega \Big) \Bigg) = (\xi^1_j - \xi^0_j) \lrcorner
\partial \lk \overline{\partial} (\xi_i^1 -f_i ) \lrcorner \omega
\rk,  \]
\[ \bar{\partial} \Bigg( (\xi^1_i - \xi^0_i) \lrcorner \partial \Big( (\xi^1_j - f_j)
\lrcorner \omega \Big) \Bigg) =  (\xi^1_i - \xi^0_i) \lrcorner
\partial \lk \overline{\partial} (\xi_j^1 -f_j ) \lrcorner \omega
\rk, \] since we observe $\overline{ \partial } \xi^0_i = \overline
{
\partial } \xi^1_i$ on $U_0 \cap U_1$, and the last step stems from
our choice of $f_i$.


\end{document}